\documentclass[reqno,11pt]{amsart}
\usepackage[utf8]{inputenc}
\calclayout
\usepackage{amsthm,amsfonts,amstext,amssymb,amsmath,latexsym,mathtools} 
\usepackage{enumerate}
\usepackage{bm}
\usepackage[abs]{overpic} 
\usepackage{cite} 
\usepackage{mathrsfs}
\DeclareMathAlphabet{\mathpzc}{OT1}{pzc}{m}{it}
\usepackage{parskip}

\usepackage{hyperref} 
\usepackage{comment}

\usepackage{listings}
\usepackage{tikz}
\usetikzlibrary{arrows}
\usetikzlibrary{calc}

\usepackage[font=small]{caption}
\usepackage[labelformat=simple]{subcaption}

\usepackage{color}
\usepackage{esint} 
\usepackage{soul} 
\usepackage[colorinlistoftodos,prependcaption,textsize=small]{todonotes}

\definecolor{red}{RGB}{255,0,0}
\definecolor{green}{RGB}{0,100,0}
\definecolor{blue}{RGB}{0,0,255}  

\usepackage{cleveref}
 
\crefname{equation}{equation}{equations}
\crefname{figure}{Figure}{Figures}

\theoremstyle{plain}
\newtheorem{thm}{Theorem}[section]

 \newtheorem{thmx}{Theorem}
\newtheorem{prop}[thm]{Proposition}
\newtheorem{lem}[thm]{Lemma}

\theoremstyle{definition}
\newtheorem{definition}[thm]{Definition} 
\newtheorem{example}[thm]{Example}

\theoremstyle{remark}
\newtheorem{remark}[thm]{Remark} 
\numberwithin{equation}{section} 

\numberwithin{equation}{section}

\renewcommand{\Re}{\mathop{\rm Re}}
\renewcommand{\Im}{\mathop{\rm Im}}
\newcommand{\K}{\mathbb K}
\renewcommand{\P}{\mathbb P}

\DeclareMathOperator{\diag}{diag}

\newcommand{\bbD}{\mathbb D}

\newcommand{\bbT}{\mathbb T}

\newcommand{\mcg}{{\mathcal G}}
\newcommand{\mcl}{{\mathcal L}}
\newcommand{\mcm}{{\mathcal M}}

\DeclareMathOperator{\rank}{\mbox{rank}}
\newcommand{\bbN}{\mathbb{N}}
\newcommand{\N}{\mathbb{N}}
\newcommand{\C}{\mathbb{C}}
\newcommand{\R}{\mathbb{R}}
\newcommand{\Z}{\mathbb{Z}}
\newcommand{\D}{\mathbb{D}}
\newcommand{\T}{\mathbb{T}}

\newcommand{\cmv}[1]{\bm{\mathcal{G}}^{(#1)}}

\newcommand\restr[2]{{
		\left.\kern-\nulldelimiterspace 
		#1 
		\vphantom{\big|} 
		\right|_{#2} 
}}

\renewcommand{\ell}{l}

\usepackage{tikz}
\usetikzlibrary{patterns}
\usetikzlibrary{decorations.pathmorphing}
\usetikzlibrary{decorations.pathreplacing,decorations.markings}

\tikzstyle{vecArrow} = [thick, decoration={markings,mark=at position
   1 with {\arrow[semithick]{open triangle 60}}},
   double distance=1.4pt, shorten >= 5.5pt,
   preaction = {decorate},
   postaction = {draw,line width=1.4pt, white,shorten >= 4.5pt}]
\tikzstyle{innerWhite} = [semithick, white,line width=1.4pt, shorten >= 4.5pt]

\makeatletter
\def\grd@save@target#1{%
  \def\grd@target{#1}}
\def\grd@save@start#1{%
  \def\grd@start{#1}}
\tikzset{
  grid with coordinates/.style={
    to path={%
      \pgfextra{%
        \edef\grd@@target{(\tikztotarget)}%
        \tikz@scan@one@point\grd@save@target\grd@@target\relax
        \edef\grd@@start{(\tikztostart)}%
        \tikz@scan@one@point\grd@save@start\grd@@start\relax
        \draw[minor help lines] (\tikztostart) grid (\tikztotarget);
        \draw[major help lines] (\tikztostart) grid (\tikztotarget);
        \grd@start
        \pgfmathsetmacro{\grd@xa}{\the\pgf@x/1cm}
        \pgfmathsetmacro{\grd@ya}{\the\pgf@y/1cm}
        \grd@target
        \pgfmathsetmacro{\grd@xb}{\the\pgf@x/1cm}
        \pgfmathsetmacro{\grd@yb}{\the\pgf@y/1cm}
        \pgfmathsetmacro{\grd@xc}{\grd@xa + \pgfkeysvalueof{/tikz/grid with coordinates/major step}}
        \pgfmathsetmacro{\grd@yc}{\grd@ya + \pgfkeysvalueof{/tikz/grid with coordinates/major step}}
        \foreach \x in {\grd@xa,\grd@xc,...,\grd@xb}
        \node[anchor=north] at (\x,\grd@ya) {\pgfmathprintnumber{\x}};
        \foreach \y in {\grd@ya,\grd@yc,...,\grd@yb}
        \node[anchor=east] at (\grd@xa,\y) {\pgfmathprintnumber{\y}};
      }
    }
  },
  minor help lines/.style={
    help lines,
    step=\pgfkeysvalueof{/tikz/grid with coordinates/minor step}
  },
  major help lines/.style={
    help lines,
    line width=\pgfkeysvalueof{/tikz/grid with coordinates/major line width},
    step=\pgfkeysvalueof{/tikz/grid with coordinates/major step}
  },
  grid with coordinates/.cd,
  minor step/.initial=.2,
  major step/.initial=1,
  major line width/.initial=0.25mm,
}

\pgfdeclarepatternformonly{north east lines wide}%
   {\pgfqpoint{-1pt}{-1pt}}%
   {\pgfqpoint{10pt}{10pt}}%
   {\pgfqpoint{5pt}{5pt}}%
   {
     \pgfsetlinewidth{.8pt}
     \pgfpathmoveto{\pgfqpoint{0pt}{0pt}}
     \pgfpathlineto{\pgfqpoint{9.1pt}{9.1pt}}
     \pgfusepath{stroke}
    }

\tikzset{
  on each segment/.style={
    decorate,
    decoration={
      show path construction,
      moveto code={},
      lineto code={
        \path [#1]
        (\tikzinputsegmentfirst) -- (\tikzinputsegmentlast);
      },
      curveto code={
        \path [#1] (\tikzinputsegmentfirst)
        .. controls
        (\tikzinputsegmentsupporta) and (\tikzinputsegmentsupportb)
        ..
        (\tikzinputsegmentlast);
      },
      closepath code={
        \path [#1]
        (\tikzinputsegmentfirst) -- (\tikzinputsegmentlast);
      },
    },
  },
  mid arrow/.style={postaction={decorate,decoration={
        markings,
        mark=at position .5 with {\arrow[#1]{stealth}}
      }}},
  end arrow/.style={postaction={decorate,decoration={
        markings,
        mark=at position 1 with {\arrow[#1]{stealth}}
      }}},
  start arrow/.style={postaction={decorate,decoration={
        markings,
        mark=at position 0 with {\arrow[#1]{stealth}}
      }}},
}

\makeatother
\tikzset{every state/.style={minimum size=0pt}}

\title[Projective geometry, matrices and OPUC]{Poncelet--Darboux, Kippenhahn, and Szeg\H{o}: interactions between projective geometry, matrices and orthogonal polynomials}

\author[M.~Hunziker]{Markus Hunziker}

\address[MH]{Department of Mathematics, Baylor University, Waco TX, USA}

\email{Markus\_Hunziker@baylor.edu}

\author[A. Mart\'{\i}nez-Finkelshtein]{Andrei Mart\'{\i}nez-Finkelshtein}

\address[AMF]{Department of Mathematics, Baylor University, Waco TX, USA, and Department of Mathematics, University of Almer\'{\i}a, Almer\'{\i}a, Spain}

\email{A\_Martinez-Finkelshtein@baylor.edu}

\author[T.~Poe]{Taylor Poe}

\address[TP]{Department of Mathematics, Baylor University, Waco TX, USA}

\email{Taylor\_Thompson8@baylor.edu}

\author[B.~Simanek]{Brian Simanek}

\address[BS]{Department of Mathematics, Baylor University, Waco TX, USA}

\email{Brian\_Simanek@baylor.edu}

\date{\today}
 
 

\begin{document}

\begin{abstract} 
 
We study algebraic curves that are envelopes of families of polygons supported on the unit circle $\bbT$.  We address, in particular, a characterization of such curves of minimal class and show that all realizations of these curves are essentially equivalent and can be described in terms of orthogonal polynomials on the unit circle (OPUC), also known as Szeg\H{o} polynomials.  Our results have connections to classical results from algebraic and projective geometry, such as theorems of Poncelet, Darboux, and Kippenhahn; numerical ranges of a class of matrices; and Blaschke products and disk functions.

This paper contains new results, some old results presented from a different perspective or with a different proof, and a formal foundation for our analysis.  We give a rigorous definition of the Poncelet property, of curves tangent to a family of polygons, and of polygons associated with Poncelet curves.  As a result, we are able to clarify some misconceptions that appear in the literature and present counterexamples to some existing assertions along with necessary modifications to their hypotheses to validate them.  For instance, we show that curves inscribed in some families of polygons supported on $\bbT$ are not necessarily convex, can have cusps, and can even intersect the unit circle.

Two ideas play a unifying role in this work.  The first is the utility of OPUC and the second is the advantage of working with tangent coordinates.  This latter idea has been previously exploited in the works of B. Mirman, whose contribution we have tried to put in perspective.
\end{abstract}

\maketitle


\section{Introduction} \label{sec:intro}

The main theme of this   paper is the study of real plane algebraic curves in the unit disk $\bbD:=\{z\in \C:\, |z|<1\}$ that can be inscribed in families of polygons with vertices on the unit circle 
$\bbT:=\partial \bbD=\{z\in \C:\, |z|=1\}$. In our discussion below we use terminology that is carefully explained in the paper, especially in Section~\ref{sec:background}.
The purpose of this introduction is to give a brief historical overview and a summary of  recent results and our contributions.
 
In 1746,  the English surveyor    W. Chapple discovered that  a circle $C$ in $\bbD$ with center at $a\in \bbD$ is inscribed in infinitely many triangles with vertices on $\bbT$ if and only if its radius is
\begin{equation}\label{eq:Chapple}
r=\frac{1-|a|^2}{2}\ .
\end{equation}
Even though Chapple's proof was flawed (see e.g. \cite{DelCentina:2016a}), it is clear that he understood that the existence 
of one triangle inscribed in $\mathbb T$ and circumscribed about $C$ implies  the existence of an infinite family of  such triangles. Chapple's formula  (\ref{eq:Chapple})
is then easily obtained by looking at an isosceles triangle in the family.

This brings us to 
\emph{Poncelet's closure theorem}, also known as \emph{Poncelet's porism} (see e.g.~\cite{DelCentina:2016a, DelCentina:2016b}), one of the most surprising and beautiful results of planar projective geometry. The theorem was discovered by J.-V. Poncelet in 1813 while he was a prisoner of war in Russia and published in 1822 in his \emph{Trait\'e sur les propri\'et\'es projectives des figures} \cite{Poncelet:1822}.
In slightly simplified form, adjusted to  our context, the result can be stated as follows:
\begin{thmx}[Poncelet, 1813] \label{thm:Poncelet}
Let   $C$ be an ellipse in $\mathbb D$, and suppose there is an $n$-sided polygon $\mathscr{P}_0$ inscribed in $\mathbb T$ and circumscribed about $C$. Then for any point $z\in \mathbb T$, there exists an $n$-sided polygon $\mathscr{P}(z)$ inscribed in $\mathbb{T}$ and circumscribed about $C$ 
such that  $z$ is a vertex of $\mathscr{P}(z)$.
\end{thmx}

\begin{figure}[htb]
	\begin{center}
	\includegraphics[width=0.35\linewidth]{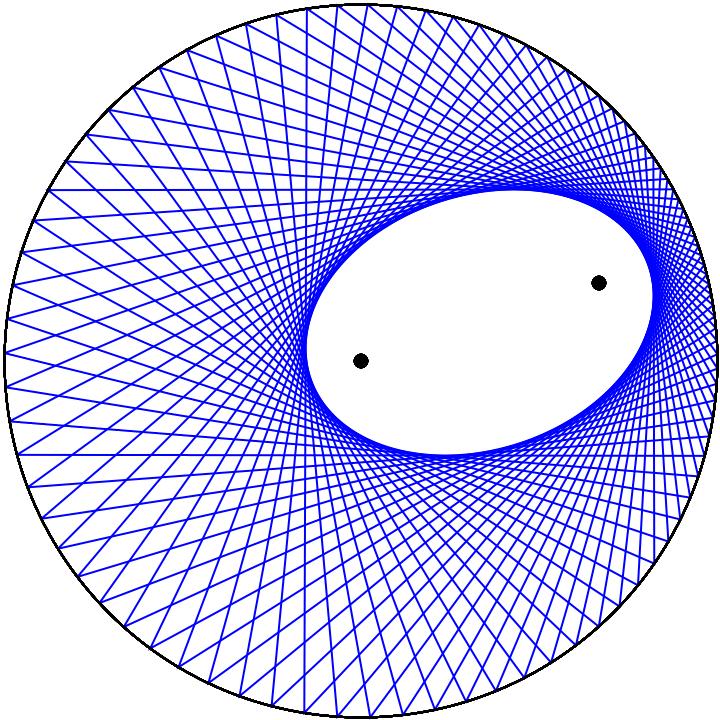}
	\end{center}
	\caption{An ellipse with a Poncelet property.}
	\label{fig:PonceletEllpse}
\end{figure} 

Here by an $n$-sided polygon we mean a simple closed polygonal chain of length $n$. However,  Poncelet's closure theorem remains true if we replace $n$-sided polygons inscribed in $\mathbb T$ by (possibly self-intersecting) closed polygonal chains of length $n$ with vertices on $\mathbb T$.
In either case, we say that an ellipse as in the theorem has the $n$-\emph{Poncelet} property.

After Poncelet published his \emph{Trait\'e}, following a suggestion of J.~Steiner, C.~G.~J.~Jacobi used the theory of elliptic integrals to give a proof of Poncelet's closure theorem in the special  case when  $C$ is a circle. Inspired by Jacobi's work,  N.~Trudi and later A.~ Cayley used elliptic integrals to prove  Poncelet's  theorem in the general case (see e.g. \cite[Chs. 3--5]{DelCentina:2016a}). Furthermore, Cayley was able to find an explicit  criterion for ellipses to have the $n$-Poncelet property in terms of power series expansions of the square-root of certain determinants. Cayley's criterion makes it possible to determine, for example, 
for which semi-minor  axes an ellipse with prescribed foci has the $n$-Poncelet property. A modern interpretation of Cayley's criterion was given by Griffith and Harris \cite{Griffiths:1978kw}.

Are there any algebraic curves $C$  in $\mathbb D$ of higher degree that have the $n$-Poncelet property?
In the late 1860's, G.~ Darboux started to investigate this question, and he realized that it is best approached by considering
the dual problem of finding curves passing through the intersection points of $n$ tangent lines to $\mathbb T$. By introducing a
convenient system of plane coordinates, now called Darboux coordinates, he was able to give a new proof of Poncelet's closure theorem and generalizations (see e.g. \cite[Ch.~9]{DelCentina:2016a}.)
Adjusted to  our context, Darboux's results can be summarized as follows:

\begin{thmx}[Darboux \cite{Darboux:1917}, 1917] \label{thm:Darboux}
Let  $C$  be  a closed  convex curve in $\mathbb{D}$  and suppose that there is an $n$-sided polygon $\mathscr{P}_0$    inscribed in $\mathbb{T}$
and  circumscribed about $C$. 
If the curve $C$ is a connected component of a real algebraic curve $\Gamma$ in  $\mathbb{D}$ of class\footnote{The class of a plane algebraic curve is the degree of its dual curve, see Section~\ref{sec:geometry} for details.}
 $n-1$ such that  each diagonal of  $\mathscr{P}_0$  is tangent to $\Gamma$,
then for every  point $z \in \mathbb{T}$,  there exists an $n$-sided polygon $\mathscr{P}(z)$ inscribed in $\mathbb{T}$ and circumscribed about $C$ 
such that  $z$ is a vertex of $\mathscr{P}(z)$ and each diagonal of $\mathscr{P}(z)$ is tangent to $\Gamma$.
In the special case when $C$ is an ellipse,  there always exists such an algebraic curve $\Gamma$   and this curve decomposes  into $(n-1)/2$ ellipses if $n$ is odd, and  $(n-2)/2$ ellipses and an isolated point if $n$ is even. 
\end{thmx}

If $C$ and $\Gamma$ are as in the theorem, then $\Gamma$ can be recovered from $C$ as the envelope of all the diagonals of 
the family of $n$-sided polygons  inscribed in $\mathbb{T}$ and  circumscribed about $C$
(see Figure~\ref{fig:Packet}). This will be made precise in Section~\ref{sec:evolvents} in terms of Darboux's  curve of degree $n-1$ that 
is dual to $\Gamma$.  We note that  even though the curve $\Gamma$ is
singular in general, it has exactly $n-1$  tangent lines in each arbitrarily given direction (just as in the special case when $\Gamma$ decomposes into ellipses).

\begin{figure}[htb] 
	\begin{center}
		\begin{tabular}{c p{3mm}c}		
			\includegraphics[width=0.35\linewidth]{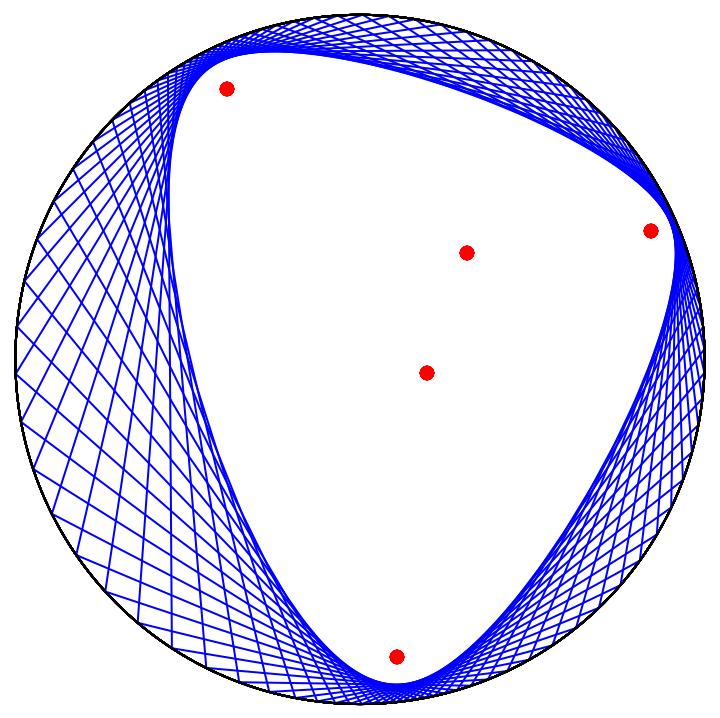} & 	& \includegraphics[width=0.35\linewidth]{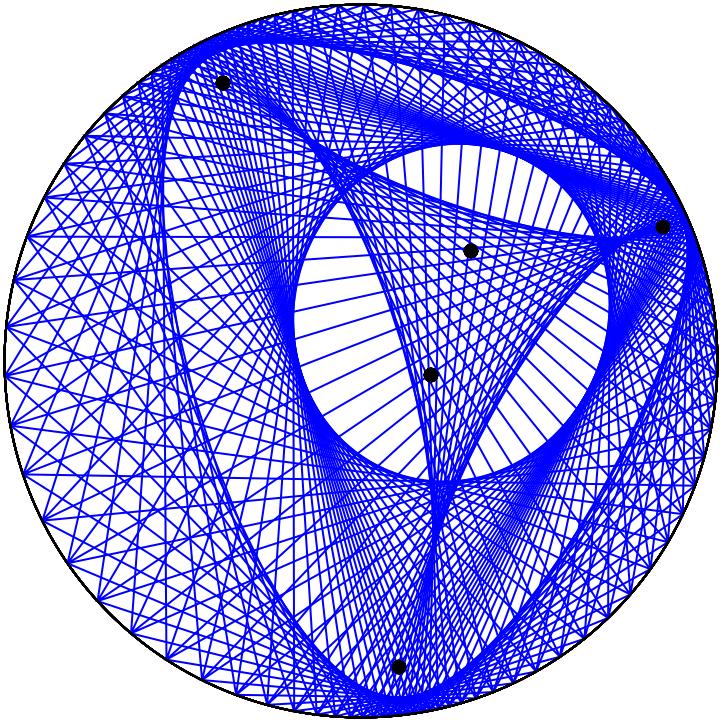} 
		\end{tabular}
	\end{center}
	\caption{Envelopes of convex hulls (left) and of all closed polygons with vertices at a family of points on the circle.}
	\label{fig:Packet}
\end{figure}

In the  study of conic sections, the idea of  foci plays a central role. In 1832, J.~Pl\"ucker   defined foci of higher degree algebraic curves  (see the definition in Section~\ref{sec:geometry}). It turns out that a curve $\Gamma$ as in the theorem above has exactly $n-1$ real foci (counted with multiplicity) and $(n-1)(n-2)$ nonreal foci. Furthermore, all the real foci are in $\mathbb D$. We will later see that the $n-1$ real foci determine $\Gamma$ completely. A priori, this fact is by no means obvious and was not 
discovered until twenty years ago.

One way to find an algebraic curve with given foci is using the notion of the numerical range of a matrix (see the definition in Section~\ref{sec:numrange}). 
\begin{thmx}[Kippenhahn \cite{MR59242, MR2378310}] \label{thm:Kippenhahn}
For an $n\times n$ complex matrix $\bm A$ there exists a real  algebraic curve $\Gamma$ of class $n$ whose real foci are the eigenvalues of $\bm A$ and such that the numerical range of $ \bm A$ is the convex hull of the real points of $\Gamma$.
\end{thmx}
Kippenhahn's proof is constructive and elegant, and we will explain  his arguments briefly in Section~\ref{sec:numrange}.

Starting in 1998, H.-L. Gau and P. Y. Wu  and, independently, B.~Mirman  
studied when the boundary of the  numerical range of an $n\times n$ matrix $\bm A$ exhibits the $(n+1)$-Poncelet property with respect to the unit circle 
$\mathbb T$.
In a series of  papers, Gau and Wu \cite{Gau:2003dw, Gau:1998bu, Gau:1998fk, Wu:2000bn} showed that a necessary and sufficient condition is that $\bm A$ belongs to the class $\mathcal{S}_{n}$ of completely non-unitary contractions with defect index 1. This class can also be identified with the compressed multiplication operators on $\bbT$.  The corresponding Poncelet polygons will join the eigenvalues of all rank 1 unitary extensions of $\bm A\in \mathcal{S}_{n}$.
Mirman, using slightly different terminology and more geometric techniques,  gave beautiful new proofs of many of  Darboux's results, apparently without being aware of Darboux's work until 2005 (see the comments in the introduction to \cite{Mirman:2005dj}).

A seemingly alternative approach to the problem above, using rational functions, can be traced back to the following result of J.~Siebeck, popularized in Marden's 1948 book \cite[Ch.~1, \S 4]{Marden66}, \emph{Geometry of Polynomials}:
\begin{thmx}[Siebeck   \cite{Sieb}]   Let $\{w_1, \ldots,w_{n}\}$ be the vertices of a convex polygon in $\C$ ordered counter-clockwise.   Let
	\begin{equation}\label{1.1}
		M(z)=\sum_{j=1}^{n} \frac{m_j}{z-w_j},
	\end{equation}
	where $m_1,\ldots,m_n$ are   real numbers.
	Then the zeros of $M(z)$ are the foci of a real algebraic curve of class $n-1$ which intersects each of the line segments $[w_j,w_k]$, $j\not=k$, at the point dividing the line in ratio $m_j/m_k$.
\end{thmx}
Daepp, Gorkin and collaborators \cite{Daepp:2002km, MR3932079, Daepp:2010fz, Gorkin:2017bv} published a series of papers concerning finite Blaschke products $b(z)$. These are the Schur functions (analytic maps of $\bbD$ to itself) which are analytic in a neighborhood of $\overline{\bbD}$, of magnitude $1$ on $\partial\bbD$, with $n$ zeros in $\bbD$.  A connection with Siebeck's theorem is through the formula (see \cite{MR3932079})
$$
    \sum_{j=1}^{n } \frac{m_j(\lambda)}{z-w_j} =\frac{b(z)}{zb(z)-\bar{\lambda}},
$$
which shows that solutions of equations of the form
\begin{equation}\label{eqBl}
zb(z)=\bar{\lambda} \in \bbT
\end{equation}
generate configurations of points on $\bbT$ such that the envelope of the convex hull of these points is a (component) of an algebraic curve whose foci coincide with the zeros of $b$. Furthermore, Fujimura \cite{Fujimura:2013ct, Fujimura:2017jgb, Fujimura:2018dt} extended their analysis to the whole interior curve and its dual, see details in Section~\ref{sec:poncelet_general}. 

It turns out that both approaches (via the numerical range or using Blaschke products) are  equivalent, and in this sense, provide the only possible construction of algebraic curves in $\bbD$ with prescribed foci having a Poncelet property. In a recent work \cite{MR3945586}, some of the authors of this paper showed that both points of view are naturally connected via orthogonal polynomials on the unit circle (OPUC), initially studied in a systematic way by Szeg\H{o} \cite{szego:1975} and Geronimus \cite{MR0133642}; for a modern account on the theory, see the treatise by Simon \cite{Simon05b}. In particular, it was shown that the class $\mathcal{S}_{n}$ is exactly the class of the truncated CMV matrices that are the natural family of matrices to be studied in the theory of OPUC and that equations \eqref{eqBl} define the zeros of the so-called \emph{paraorthogonal} polynomials on the unit circle.  

Poncelet's construction can be approached also from a point of view of the theory of elliptic billiards \cite{Dragovic:2014ve, Dragovic:2019iy, MR3519751, MR2994045} and discrete dynamical systems such as the pentagram map \cite{Schwartz:1992, Schwartz:2015cn, Ovsienko:2010kj, MR4030385}. 

The literature on Poncelet is vast (so the bibliography cited here is far from complete), with a clear increase in attention to the subject in the last two decades. However, trying to navigate through these alternative and complementary perspectives can be also frustrating due to diversity of terminology and, to be said, occasional lack of rigor. Notions like a curve inscribed in a family of lines or an envelope of polygons, or even a Poncelet curve, are many times left to intuition and can lead (and did lead) to wrong conclusions. It also became clear to us that the phenomenon discovered by Poncelet is a ``shell'' (convex hull) of a much richer structure analyzed by Darboux (and incidentally, rediscovered by Mirman) and should be studied from that perspective. Thus, our initial motivation was to unify several results scattered in the literature, in part using the tool of OPUC, as well as to define precisely the objects we are working with. This paper is the first result of this endeavor; it contains several results, some old (either stated without proof or proved from a different perspective), some new. In particular,
\begin{itemize}
	\item We give a rigorous definition of the Poncelet property, of curves tangent to a family of polygons (or envelopes of families of polygons),  and polygons associated with Poncelet curves.
	\item We illustrate some misconceptions that appear in literature and present counterexamples. We also analyze sufficient conditions that invalidate these situations. For instance, we show that curves inscribed in some families of polygons supported on $\bbT$ are not necessarily convex, can have cusps and even can intersect the unit circle. 
	\item We prove a characterization of all ``complete Poncelet curves'' of minimal class and show that all realizations of these curves, mentioned in this Introduction, are equivalent.
	\item Regarding the tools, OPUC, brought up in \cite{MR3945586}, play a unifying role in our analysis. We also put forward the advantage of working with tangent coordinates. They have been exploited in the works of B.~Mirman (whose contribution, in our opinion, is not sufficiently known and we tried to put in perspective), so that we call its application in our context the ``Mirman's parametrization''. Another very convenient tool from  projective geometry is reciprocation with respect to the unit circle, which turned out to be very useful, see e.g.~\cite{Fujimura:2018dt}.
\end{itemize} 
 
The structure of the paper is as follows. In order to make this paper self-contained and facilitate its reading, in Section~\ref{sec:background} we introduce three main components (``toolboxes'') of the necessary background.  Section~\ref{sec:evolvents} contains, among other things, a rigorous definition of a Poncelet curve, an analysis of when such a curve belongs to the unit disk, as well as Mirman's  parametrization of a package of Poncelet curves, along with a number of interesting examples.  Section~\ref{sec:poncelet_general} contains a characterization of complete Poncelet curves of minimal class and the equivalence of their constructions, unifying in a certain sense the previous work of  Daepp, Gorkin, Gau, Wu, Fujimura and collaborators.  Finally, in Section~\ref{sec:Mirman} we explore in more detail the setting related to Darboux's result (Theorem~\ref{thm:Darboux}), namely when one of the components of the complete Poncelet curve is an ellipse.

\section{Our toolbox} \label{sec:background}

\subsection{Tool set 1: Projective algebraic geometry} \label{sec:geometry}
 
 \
 
We start with a few elementary notions from projective algebraic geometry, which  we will need throughout this paper. All these results are standard and can be found in practically any text on classical algebraic geometry, see e.g.~\cite{MR2549804, MR2724360, MR1934359, MR2815937, MR2791970, MR3100243}.

\subsubsection{The projective plane} \label{sec:basics}
\

For any field $\K$ such as $\R$ or $\C$, the \emph{projective plane} $\P^2(\K)$  over $\K$ is the set of all equivalence classes of   triples $(x_1,x_2,x_3)\in \K^3\setminus\{(0,0,0)\}$, where  $(x_1,x_2,x_3)$ and $(x_1',x_2',x_3')$ are equivalent if and only if $(x_1',x_2',x_3')=(\lambda x_1, \lambda x_2, \lambda x_3)$ for some  $\lambda \in \K\setminus \{0\}$. The equivalence class of a  triple $(x_1,x_2,x_3)\in \K^3\setminus \{(0,0,0)\}$ is denoted $(x_1:x_2:x_3)$ and is called the point of $\P^2(\K)$ with  \emph{homogenous coordinates} $(x_1,x_2,x_3)$.  As usual, we  embed the affine plane $\K^2$ in $\P^2(\K)$ by identifying the point $(x_1,x_2)\in \K^2$ with the point  $(x_1:x_2:1)\in \P^2(\K)$ and conversely, any point $(x_1:x_2:x_3)\in \P^2(\K)$ such that $x_3 \neq 0$ with the point $(x_1 / x_3, x_2/x_3)\in \K^{2}$. The complement of $\K^2$ in $\P^2(\K)$, that is, the set $\{(x_1: x_2 : x_3) \in \P^2(\K) : x_3=0\}$, is called the \emph{line at infinity}.
 
The \emph{real projective plane} $\P^2(\R)$ is canonically embedded in   \emph{complex projective plane}  $\P^2(\C)$.
Furthermore, for much of this paper, we will identify $\R^2$ with   $\C$   and  hence  $\C$ is embedded in    $\P^2(\R)$  as above:
\begin{equation}\label{embedding C in RP^2}
 \C=\R^2  \subset \P^2(\R), \quad x+iy=(x,y) \mapsto (x:y:1).
 \end{equation}
We view $\P^2(\R)$ as  a   real two-dimensional compact manifold in the usual way so that  (\ref{embedding C in RP^2}) is an open embedding.
The image of this embedding is  dense in $\P^2(\R)$, and hence $\P^2(\R)$ is a compactification of $\mathbb C$. This compactification   is not to be confused 
with the one-point compactification of $\C$ which is the Riemann sphere.
 
 \subsubsection{Real algebraic curves} \label{sec:basics}
 \
 
A \emph{plane real  algebraic curve}  of degree $d$  is an algebraic curve $\Gamma\subset \P^2(\C)$ 
defined by an equation $F(x_1,x_2,x_3)=0$, where
 $F(x_1,x_2,x_3)$ is a nonzero homogeneous polynomial  of degree $d$ with \emph{real} coefficients. 
Notice  that $F(x_1,x_2,x_3)$ being zero only depends on the equivalence class of the triplet $(x_1,x_2,x_3)$ since $F(\lambda x,  \lambda y, \lambda z)=\lambda^d F(\lambda x,  \lambda y,\lambda z)$, and that although we speak about a real curve $\Gamma$, we have $\Gamma\subset \P^2(\C)$. We say that the curve $\Gamma$ is \emph{irreducible} if the polynomial $F(x_1,x_2,x_3)$ is irreducible over $\C$.

The \emph{set of real points} of a real algebraic curve $\Gamma\subset \P^2(\C)$ is defined as the set
$$
\Gamma(\R):=\Gamma\cap \P^2(\R).  
$$
In this paper, we will be very careful to always distinguish between $\Gamma$ and $\Gamma(\R)$.
 
If $f(x,y)$ is any (not necessarily homogenous) polynomial of degree $d$ with real coefficients, then  
$F(x_1,x_2,x_3):=x_3^d f(x_1/x_3,x_2/x_3)$ is a homogenous 
polynomial of degree $d$ with real coefficients.
Thus, a nonzero polynomial $f(x,y)$ with real coefficients  defines a real algebraic curve  $\Gamma$ via this homogenization.
Geometrically, $\Gamma$ is the   projective closure of the curve given by the equation $f(x,y)=0$, that is,  $\Gamma=\text{clo}{\{(x,y)\in \C^2 : f(x,y)=0\}}$. Conversely,    $F(x_1,x_2,x_3)$ can be dehomogenized by setting $f(x,y):=F(x,y,1)$ and we have $\Gamma  \cap \C^2=\{(x,y)\in \C^2 : f(x,y)=0\}$.

\begin{remark}
We should emphasize that  a real plane algebraic curve $\Gamma$ is more than its set of real points $\Gamma(\R)$
and even the set of complex points of $\Gamma$ does not determine its  defining polynomial $F(x_1,x_2,x_3)$.
However, if we assume that  $F(x_1,x_2,x_3)$ is square-free, that is, without any repeated  irreducible factors,
then $\Gamma$ determines $F(x_1,x_2,x_3)$ uniquely up to a nonzero multiplicative constant.  We will always make this assumption.
\end{remark}

\subsubsection{Nonsingular and singular points}

\

Suppose   $\Gamma$ is defined by the equation $F(x_1,x_2,x_3)=0$.
Then  $(a_1:a_2:a_3)\in \Gamma$  is called  a \emph{nonsingular point} of $\Gamma$ if
$$
\left(\frac{\partial F}{\partial x_1} (a_1,a_2,a_3),\frac{\partial F}{\partial x_2}  (a_1,a_2,a_3), \frac{\partial F}{\partial x_3}  (a_1,a_2,a_3)\right)\not= (0,0,0);
$$
otherwise $(a_1:a_2:a_3)$ is  a \emph{singular point}.
We say that $\Gamma$ (resp., $\Gamma(\R)$) is \emph{nonsingular}, if all points of $\Gamma$ (resp., $\Gamma(\R)$) are nonsingular.
Note that if $(a_1:a_2:a_3)$ is a nonsingular point, then the equation
\begin{equation}\label{def: tangent line}
\frac{\partial F}{\partial x_1}(a_1:a_2:a_3)\, x_1 + \frac{\partial F}{\partial x_2} (a_1:a_2:a_3)\, x_2+\frac{\partial F}{\partial x_3} (a_1:a_2:a_3)\, x_3 =0
\end{equation}
defines a line in $\P^2(\R)$ (resp., $\P^2(\C)$) which is called the \emph{tangent line} of $\Gamma(\R)$ (resp., $\Gamma$)
at the point $(a_1:a_2:a_3)$. 
If $(a_1:a_2:a_3)$  is a singular point, then the equation (\ref{def: tangent line}) is meaningless.
However, since an algebraic curve has only finitely many singular points (and hence every singular point is an isolated point),   tangent lines at a singular point 
$(a_1:a_2:a_3)$   can be defined by continuity.  
For example,
Figure~\ref{fig:tangents} shows the tangent lines at a \emph{node} (or \emph{ordinary double point}) and at a \emph{cusp} (or \emph{return point}).

\begin{figure}[htb]
	\begin{center}
		 \includegraphics[width=0.6\linewidth]{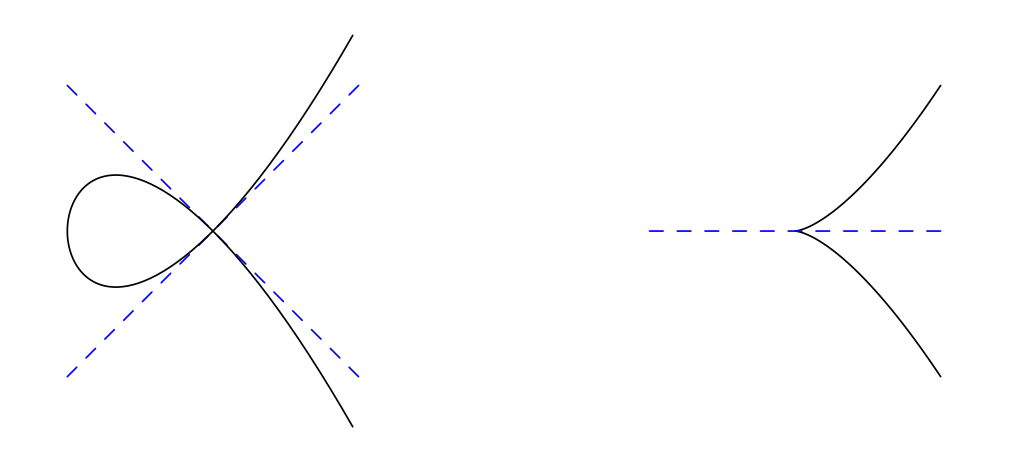}
		 	\end{center}
	\caption{Tangent lines at a node (left) and at a cusp.}
	\label{fig:tangents}
\end{figure} 

\subsubsection{Duality} \label{subsec:duality}

\

The  \emph{dual projective plane} $\P^2(\K)^*$  is the set of lines in $\P^2(\K)$.
We will identify $\P^2(\K)$ with $\P^2(\K)^*$  by letting the point $(u_1:u_2:u_3)\in \P^2(\K)$ correspond to the line in $\P^2(\K)$ given by the equation
\begin{equation}\label{eq:negative sign}
u_1x_1+u_2x_2-u_3x_3=0.
\end{equation}
Note that via this identification, we can view  $\P^2(\R)^*$ as a subset of $\P^2(\C)^*$.
The motivation for the negative sign in (\ref{eq:negative sign}) will be explained in the next subsection.
 
The \emph{dual}  of a real algebraic curve $\Gamma\subset \P^2(\C)$ is the real algebraic curve $\Gamma^*\subset \P^2(\C)$ whose points correspond to the tangent lines of $\Gamma$, that is, 
\begin{equation} \label{eq:defDual}
\Gamma^*:=\left \{(u_1:u_2: u_3)\in \P^2(\C) :\, u_1 x_1+u_2 x_2 - u_3 x_3 =0 \text{ is a  tangent line of $\Gamma$}  \right\}.\\
\end{equation}
In this context, triples $(u_1:u_2: u_3)$ are known as the  \emph{tangent coordinates} of the dual $\Gamma^*$.

Given an equation $F(x_1,x_2,x_3)=0$ that defines $\Gamma$, we can find an equation $G(u_1,u_2,u_3)=0$ that defines $\Gamma^*$
by eliminating the variables $x_1,x_2,x_3$ from the system of equations 
$$
F(x_1,x_2,x_3)=0,\ \frac{\partial F}{\partial x_1}(x_1,x_2,x_3)=u_1,\ \frac{\partial F}{\partial x_2}(x_1,x_2,x_3)=u_2,\ \frac{\partial F}{\partial x_3}(x_1,x_2,x_3)=-u_3.
$$
In practice, this can be achieved by computing a suitable Gr\"obner basis.  For example,  to obtain $G(u_1,u_2,u_3)$ from $F(x_1,x_2,x_3)$,
we can use the Mathematica code 
\begin{align*}
&\mbox{{\tt GroebnerBasis[\{F, D[F,x\textsubscript{1}]-u\textsubscript{1}, D[F,x\textsubscript{2}]-u\textsubscript{2}, D[F,x\textsubscript{3}]+u\textsubscript{3}\},}}\\
&\mbox{{\tt\ \ \ \{u\textsubscript{1},u\textsubscript{2},u\textsubscript{3}\},\{x\textsubscript{1},x\textsubscript{2},x\textsubscript{3}\}, MonomialOrder -> EliminationOrder]}},
\end{align*}
where {\tt F}  stands for the polynomial $F(x_1,x_2,x_3)$. The output will be a Gr\"obner basis with a single element, namely the desired polynomial $G(u_1,u_2,u_3)$.
All plots of dual curves in this paper were produced by first obtaining the equation of the dual in this manner.

The terminology ``dual'' is justified by the fact that $(\Gamma^*)^*=\Gamma$. The degree  of  $\Gamma^*$  is called the \emph{class} of $\Gamma$.
The relationship between the degree $d$ and  the class of $\Gamma$   is rather subtle and is  described by Pl\"ucker's formula. In the special case when $\Gamma$ is nonsingular, Pl\"ucker's formula says that the class of $\Gamma$ is equal to $d(d-1)$. Since $(\Gamma^*)^*=\Gamma$, 
this means that if $\Gamma$ is nonsingular and $d>2$, then $\Gamma^*$ must have singular points.
The dual of nonsingular conic ($d=2$) is again a nonsingular conic; the dual of a line ($d=1$) is a point.

We will be mostly interested in the set of real points of the duals  of real algebraic curves.
Suppose  $C\subseteq  \Gamma(\R)$ is a  union of path-components of the set of real points of a real algebraic curve $\Gamma$. Then we define the dual of $C$ as the set
\begin{equation}
C^*:=\left \{(u_1:u_2:u_3)\in \P^2(\R):\, u_1 x_1+u_2 x_2 - u_3 x_3 =0 \text{ is a  tangent line of $C$}  \right\}.
\end{equation}
In particular, $C^* \subseteq \Gamma^*(\R)$.
Furthermore,  if $C$ is a path-component of $\Gamma(\R)$, then $C^*$ is a path-component of $\Gamma^*(\R)$ and  $(C^*)^*=C$.

\subsubsection{Reciprocation about $\bbT$}

\

A nice geometric interpretation of duality in the real projective plane $\P^2(\R)$  can be given in terms of so-called reciprocation about the unit circle $\bbT$.
In the following, we will view $\C\subset \P^2(\R)$ as in (\ref{embedding C in RP^2}). Then  $\bbT$ is the set of real points
of the algebraic curve given by $x_1^2+x_2^2-x_3^2=0$.

The  \emph{reciprocal} or \emph{polar} of a   point $\zeta=u+iv \not =0$ in $\C$  about  $\bbT$ is the line $\ell$  that contains the point  $\zeta/|\zeta|^2$ 
and is perpendicular to the ray from $0$ through  $\zeta$ (see Figure~\ref{fig:Fujimura}).  We also say that $\zeta$ is the \emph{reciprocal} or the \emph{pole} of $\ell$.
The reciprocal of the origin $0$ is the line at infinity $\ell_\infty$. We can extend reciprocation about $\bbT$ to give a bijection between  points
and lines  in $\P^2(\R)$.
It is an easy exercise to verify that the reciprocal of the point $(u_1:u_1:u_3)\in \P^2(\R)$ about $\bbT$ is precisely the line  given by the equation $u_1x_1+u_2x_2-u_3x_3=0$ (which motivates the definition \eqref{eq:negative sign}--\eqref{eq:defDual}).

Suppose $l$ is a line in $\C$ that intersect the unit circle $\T$ in two distinct points $z_1$ and $z_2$. Then $l$ is 
given by the equation
\begin{equation} \label{eq: line given by two points on T}
z+z_1z_2 \overline{z} -(z_1+z_2)=0
\end{equation}
in the complex variable $z$. 
If we write $z=x+iy$ and compare (\ref{eq: line given by two points on T}) to  $ux+vy -1=0$, we   obtain 
\begin{equation} \label{correspUandVb}
u=\frac{1+z_1z_2}{z_1+z_2}\ , \quad v=i\frac{1-z_1z_2}{z_1+z_2}, \quad \zeta = \frac{2z_1z_2}{z_1+z_2}.
\end{equation}
Using  (\ref{correspUandVb}), it is now straightforward to express  the elementary symmetric functions $z_1+z_2$ and $z_1z_2$ in terms of $\zeta=u+iv$
as follows:
\begin{equation} \label{eq: elementary symmetric and zeta}
z_1+z_2 = \frac{2}{\overline{\zeta}}\ ,\quad  z_1z_2 = \frac{\zeta}{\overline{\zeta}}\ .
\end{equation}
This in turn allows us to write 
\begin{equation} \label{subst1}
	z_1= \frac{1+ i \sqrt{\zeta \overline{\zeta }-1}}{\overline{\zeta}}, \quad z_2= \frac{1- i \sqrt{\zeta \overline{\zeta }-1}}{\overline{\zeta}} \ .
\end{equation}
Geometrically, the points $z_1, z_2 \in \T$ are the points
where the two tangent lines to $\T$ containing the point $\zeta$ touch  $\T$ (see Figure~\ref{fig:Fujimura}).

\subsubsection{Mirman's parametrization} \label{sec:MirmansParam}

\

Let $\Gamma$ be a real algebraic curve   such that $\Gamma(\R)\subset \D$.
In this case, we can use reciprocation about $\T$  to give an alternative parametrization of   $\Gamma^*$ that has been systematically exploited in the work of Mirman and collaborators (see  e.g., \cite{Mirman:2003ho,Mirman:2005dj}).
Suppose $\Gamma$ has class $m$ so that
$\Gamma^*$ is given by an equation  $G(u_1,u_2,u_3)=0$, where $G(u_1,u_2,u_3)$ is a homogenous polynomial of degree $m$.
Let $g(u,v):=G(u,v,1)$ be the dehomogenization.
Using the substitution  (\ref{correspUandVb}), define a polynomial
\begin{equation} \label{MirmanSubstitution}
P(z_1,z_2):=  (z_1+z_2)^m g\left(\frac{1+z_1z_2}{z_1+z_2},i\frac{1-z_1z_2}{z_1+z_2}\right).
\end{equation}
Then $P(z_1,z_2)$ is a symmetric polynomial of (total) degree $2m$  with complex coefficients.
By our discussion above,
a point $\zeta=u+iv\in \C$ in the exterior of $\T$ lies on $\Gamma^*(\R)$ if and only if the 
points $z_1,z_2\in \T$ given by (\ref{correspUandVb}) or Figure~\ref{fig:Fujimura}  satisfy the equation
\begin{equation}\label{eq:P=0b}
P(z_1,z_2)=0.
\end{equation}
We say that in \eqref{eq:P=0b}, $\Gamma^*$ (or equivalently, $\Gamma$) is written in \textit{tangent coordinates}, and refer to it as \textit{Mirman's parametrization} of the algebraic curve.
 
\begin{figure}[htb]
	\begin{center}
		\begin{overpic}[scale=0.8]{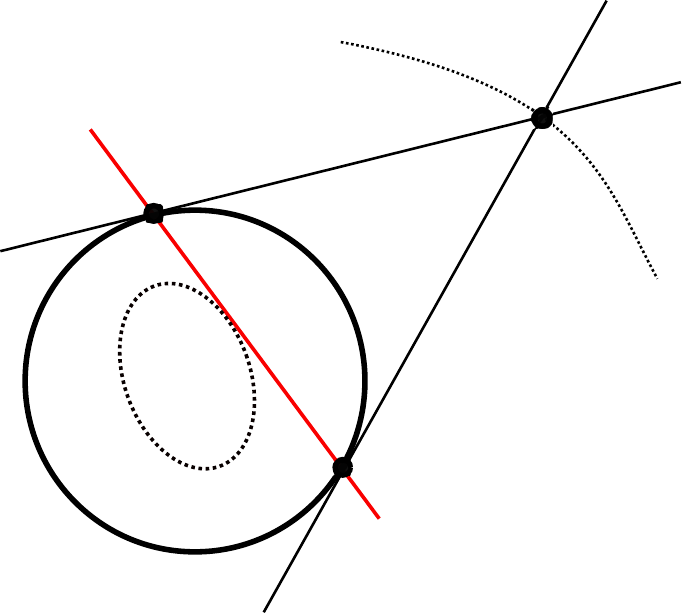}
			\put(31,98){ $z_1$}
			\put(80,30){ $z_2$}
			\put(65,50){ $\ell$}
			\put(119,103){ \small $\zeta$}
			\put(30,24){ \small $\Gamma(\R)$}
			\put(145,67){\small  $\Gamma^*(\R)$}
			\put(2,20){ $\bbT$}
		\end{overpic}
	\end{center}
	\caption{Geometric construction of the dual curve via reciprocation.}
	\label{fig:Fujimura}
\end{figure}
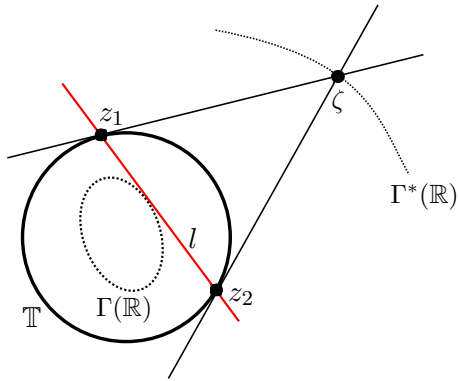

\subsubsection{Real foci} \label{sec:foci}

\

In the study of real conics (real algebraic curves of degree $2$), the points called foci have played a central role since antiquity.
Pl\"ucker generalized the concept to curves of higher degree as follows.
A point $(a_1:a_2:a_3)\in \P^2(\C)$ is called a \emph{focus} of a real algebraic curve $\Gamma\subset \P^2(\C)$  if the two lines through $(a_1:a_2:a_3)$ and the so-called circular points at infinity $(1: \pm i: 0)$ are tangent to  $\Gamma$.  A focus $(a_1:a_2:a_3)$ of $\Gamma$ is called a \emph{real focus} if $(a_1:a_2:a_3)\in \P^2(\R)$.

It is easy to verify that the lines in $\P^2(\C)$ through $(a_1:a_2:a_3)$ and $(1: \pm i: 0)$ are given by the equations $a_3x_1 \pm a_3 x_2 -(a_1\pm ia_2)x_3 =0$,
respectively. Thus, if $G(u_1,u_2,u_3)=0$ is the equation defining the dual curve $\Gamma^*$, then $(a_1:a_2:a_3)\in \P^2(\C)$
is a focus of $\Gamma$ if and only if $G(a_3, \pm i a_3, a_1\pm i a_2)=0$. If $(a_1:a_2:a_3)\in \P^2(\R)$, 
then $G(a_3, -i a_3, a_1- i a_2)=\overline{G(a_3, i a_3, a_1+ i a_2)}$ since the coefficients  of $G(u_1,u_2,u_3)$ are real. So, $(a_1:a_2:a_3)\in \P^2(\R)$ is a real focus of $\Gamma$ if and only if 
$G(a_3, i a_3, a_1+ i a_2)=0$. Viewing  $\C\subset \P^2(\R)$ as in (\ref{embedding C in RP^2}), we see that $z\in \C$ is a real focus of $\Gamma$ if and only if
\begin{equation}\label{eq:dualG}
G(1,i,z)=0.
\end{equation}

\subsection{Tool set 2: Orthogonal polynomials on the unit circle} \label{sec:blaschke}

\subsubsection{OPUC and Szeg\H{o} recursion} \label{sec:OPUC}

\

Throughout the rest of the paper we reserve the notation $\Phi_n(z)$ for monic polynomials of degree exactly $n$; 
when we need to make the dependence on the zeros explicit, we will write
\begin{equation} \label{notation:zeros}
\Phi_n(z; f_1,\dots, f_n) :=\prod_{j=1}^n \left(z-f_j \right).
\end{equation}
Moreover, if 
$$
\Phi_n(z)=\sum_{j=0}^nc_j z^j, \quad c_n=1, 
$$
then its reversed polynomial is  
\begin{equation} \label{reversed}
\Phi^*_n(z)=\sum_{j=0}^n \overline{c_j} z^{n-j}=z^n \, \overline{\Phi_n\left( 1/\overline{z} \right)}.
\end{equation}
Observe that $\Phi^*_n(z)$ can be of degree strictly less than $n$.

If $\Phi_{n}(z)=\Phi_{n}(z; f_1,\dots, f_n)$ is a monic polynomial of degree $n$ with all its zeros $f_j\in \mathbb D$, then there exists a measure $\mu$ on the unit circle $\bbT$ such that $\Phi_n(z)$ is orthogonal to $\{z^j\}_{j=0}^{n-1}$ in $L^2(\bbT,\mu)$.  There are actually many such measures $\mu$ and one example is 
$$
c\cdot |\Phi_n(e^{i\theta};f_1,\ldots,f_n)|^{-2}d\theta, 
$$
where $c$ is a normalization constant. Here we identify measures on $\bbT$ and measures on $[0,2\pi)$ in the usual way.  This means that the polynomial $\Phi_n(z)$ has all the properties of an orthogonal polynomial on the unit circle (OPUC).

The most important property of OPUC for our investigation is the \emph{Szeg\H{o} recursion}, which states that if $\Phi_k(z)$ and $\Phi_{k+1}(z)$ are consecutive orthogonal polynomials for a measure $\mu$ on $\bbT$, then there exists some $\alpha_k\in\bbD$ so that
\begin{equation} \label{szego}
\begin{pmatrix} 
\Phi_{k+1}(z)\\ \Phi_{k+1}^*(z)
\end{pmatrix}= \begin{pmatrix}
z & -\overline{\alpha_k} \\ -\alpha_k z & 1
\end{pmatrix} \, 
\begin{pmatrix}
\Phi_{k}(z)\\ \Phi_{k}^*(z)
\end{pmatrix}  . 
\end{equation}
It is easy to see that $\alpha_k=- \overline{\Phi_{k+1}(0)}$; these are known as the \emph{Verblunsky coefficients}  for the OPUC.  Notice that the Szeg\H{o} recursion can be inverted, allowing one to recover the orthogonal polynomials of degree smaller than $n$ from the degree $n$ orthogonal polynomial.  This tells us that even though there are many choices of orthogonality measure for the polynomial $\Phi_n$, they all have the same first $n$ orthogonal polynomials $\{\Phi_j(z)\}_{j=0}^{n-1}$ and hence they all have the same first $n$ Verblunsky coefficients (see \cite[Theorem 1.7.5]{Simon05b}).  Hence, any monic $\Phi_n(z)$ with all of its zeros in $\bbD$ can be alternatively parametrized by its Verblunsky coefficients $\alpha_0, \dots, \alpha_{n-1}$ (instead of its zeros).   When we need to make this dependence explicit, we will use the notation
\begin{equation} \label{notation:Verb}
\Phi_n(z)=\Phi_n^{(\alpha_0, \dots, \alpha_{n-1})}(z)
\end{equation}
in contrast to \eqref{notation:zeros}.

\subsubsection{CMV matrices} \label{sec:cmv}
\

If the orthogonality measure $\mu$ on $\mathbb T$ has infinitely many points in its support, then the sequence of its Verblunsky coefficients is also infinite.  In this case, one can define a sequence of  $2\times2$ matrices $\{\bm \Theta_j\}_{j=0}^{\infty}$ by
\[
\bm \Theta_j=\begin{pmatrix}
\bar{\alpha}_j & \sqrt{1-|\alpha_j|^2}\\
\sqrt{1-|\alpha_j|^2} & -\alpha_j
\end{pmatrix}
\]
and the operators $\bm \mcl$ and $\bm \mcm$ by
\[
\bm \mcl=\bm \Theta_0\oplus\bm \Theta_2\oplus\bm \Theta_4\oplus\cdots,\qquad\mcm=\bm 1\oplus\bm \Theta_1\oplus\bm \Theta_3\oplus\cdots
\]
where the initial $\bm 1$ in the definition of $\bm \mcm$ is a $1\times1$ identity matrix.  The \emph{CMV matrix} corresponding to $\mu$ is then given by $\bm{\mathcal G} = \bm{\mathcal G}(\{\alpha_j \}):=\bm \mcl\bm \mcm$, or explicitly,
\begin{equation} \label{cmvInf}
\bm{\mathcal G} := \begin{pmatrix}
\overline{\alpha_0} & \overline{\alpha_1} \rho_0 & \rho_1 \rho_0  & 0 & 0 & \dots \\ 
\rho_0 &- \overline{\alpha_1} \alpha_0 & -\rho_1 \alpha_0  & 0 & 0 & \dots \\
0 & \overline{\alpha_2} \rho_1 & - \overline{\alpha_2} \alpha_1   &  \overline{\alpha_3} \rho_2  & \rho_3 \rho_2 & \dots \\
0 & \rho_2 \rho_1 & - \rho_2 \alpha_1   &  -\overline{\alpha_3} \alpha_2  &- \rho_3 \alpha_2 & \dots \\
0 & 0 & 0  &   \overline{\alpha_4} \rho_3  & - \overline{\alpha_4} \alpha_3 & \dots \\
\dots & \dots & \dots  &   \dots  & \dots & \dots 
\end{pmatrix}, \quad \rho_n=\sqrt{1-|\alpha_n|^2}
\end{equation}
(see \cite[Section 4.2]{Simon05b}).  Since each of $\bm \mcl$ and $\bm \mcm$ is a direct sum of unitary matrices, each of $\bm \mcl$ and $\bm \mcm$ is unitary and hence $\bm \mcg$ is unitary as an operator on $\ell^2(\bbN)$.  The principal $n\times n$ submatrix of $\bm \mcg$, which depends only on the Verblunsky coefficients $\alpha_0, \dots, \alpha_{n-1}$, will also be called the $n\times n$ \emph{cut-off CMV matrix}, which we will denote by $\bm \mcg^{(n)}= \bm \mcg^{(n)}(\alpha_0, \dots, \alpha_{n-1})$.
The cut-off CMV matrices are the canonical representation of the compressed multiplication operator (see \cite{MR3945586}) and satisfy	
\begin{equation}
\label{characteristicequ}
\Phi_n^{(\alpha_0, \dots, \alpha_{n-1})}(z)=\det (z\bm I_n -\cmv{n}).
\end{equation}
Thus it is true that all of the eigenvalues of $\cmv{n}$ are in the unit disk $\bbD$.  Furthermore, we see from the construction that the operator norm 
$\|\cmv{n}\|=1$  and 
$$
\rank( \bm I_n-\bm \mcg^{(n)}\bm \mcg^{(n)*})=\rank(\bm I_n-\bm \mcg^{(n)*}\bm \mcg^{(n)})=1.
$$  

\subsubsection{Paraorthogonal polynomials} \label{sec:POPUC}
\

 A \emph{paraorthogonal polynomial} on the unit circle (POPUC) can be generated by the Szeg\H{o} recursion \eqref{szego} if we replace the last Verblunsky coefficient $\alpha_{n-1}$ by a value $\lambda\in \mathbb T$:
\begin{equation}
\label{def_POPUC}
\Phi_n^{(\alpha_0, \dots, \alpha_{n-2},\lambda)}(z)=  z\, \Phi_{n-1}(z) -\overline{\lambda}\,  \Phi_{n-1}^*(z), \quad \Phi_{n-1}(z)=\Phi_{n-1}^{(\alpha_0, \dots, \alpha_{n-2})}(z).
\end{equation}
The $n$ zeros $z_{j}=z_{n,j}^\lambda$, $j=1, \dots, n$, of $\Phi_n^{(\alpha_0, \dots, \alpha_{n-2},\lambda)}(z)$ are distinct and belong to $\mathbb T$.
This can be seen by noting that the zeros of $\Phi_n^{(\alpha_0, \dots, \alpha_{n-2},\lambda)}(z)$ are the eigenvalues of an $n\times n$ unitary dilation of $\cmv{n-1}$.  By this we mean that one can take the cutoff CMV matrix corresponding to the Verblunsky coefficients $\{\alpha_j\}_{j=0}^{n-2}$ and add one row and one column to get a unitary $n\times n$ matrix whose characteristic polynomial is $\Phi_n^{(\alpha_0, \dots, \alpha_{n-2},\lambda)}(z)$:
\begin{equation}
	\label{characteristicParaorthog}
	\Phi_n^{(\alpha_0, \dots, \alpha_{n-2},\lambda)}(z)=\det (z\bm I_n -\cmv{n}), \quad \cmv{n}=\cmv{n}(\alpha_0, \dots, \alpha_{n-2},\lambda), \quad \lambda \in \bbT.
\end{equation}
  In fact it is true that all $n\times n$ unitary dilations of $\cmv{n-1}$ are obtained this way by an appropriate choice of $\lambda\in\bbT$.

\begin{definition} \label{def:popuc}
If in \eqref{def_POPUC}, $\Phi_{n-1}(z)=\Phi_{n-1}(z; f_1,\dots, f_{n-1})$,  that is, if $f_1,\dots, f_{n-1}$ are the zeros of $\Phi_{n-1}(z)$, then  the $1$-parametric family of points $\mathcal{Z}_n^\lambda  = \{z_{n,1}^\lambda , \dots , z_{n,n}^\lambda \}$ on $\bbT$ is called the \emph{paraorthogonal extension} of the zeros $f_1,\dots, f_{n-1}$ of $\Phi_{n-1}(z)$.
\end{definition}
It is known that two sets, $\mathcal{Z}_n^{\lambda_1} $ and $\mathcal{Z}_n^{\lambda_2} $ from this extension, with $\lambda_1, \lambda_2\in \bbT$, $\lambda_1\neq \lambda_2$, determine the original points $f_1,\dots, f_{n-1}$, and hence 
$\Phi_{n-1}(z)$, completely (Wendroff's Theorem for OPUC, see e.g.~\cite{MR3945586}). 

For further details on orthogonal polynomials on the unit circle see e.g.~the modern treatise \cite{Simon05b}, or the more classical texts \cite{MR0133642} and \cite{szego:1975}.

\subsubsection{Blaschke products} \label{sec:BP}
\

Very much related with OPUC is the notion of Blaschke products. A normalized \emph{Blaschke product} is a rational function of the form 
\begin{equation} \label{notation:Blaschke}
	b_n(z; f_1,\dots, f_n)=   \frac{ \Phi_{n}(z)}{\Phi_{n}^*(z)},
\end{equation}
where points $f_1, f_2,\dots, f_{n-1}$ are inside the unit disk $\mathbb D$,
\[
\Phi_{n}(z)= \Phi_{n}(z; f_1,\dots, f_n),\qquad\qquad\Phi_{n}^*(z)= \Phi_{n}^*(z; f_1,\dots, f_n),
\]
see notation \eqref{notation:zeros}--\eqref{reversed}. We will say that the Blaschke product $b_n(z)$ has \emph{degree $n$}. 

For any $\lambda\in \mathbb T$, the equation
\begin{equation} \label{solutionBlaschkeB}
	b_{n}(z) =\overline{\lambda}
\end{equation}
has exactly $n$ solutions,  $z_1^\lambda , \dots , z_n^\lambda$,   all distinct and on $\mathbb T$. 
\begin{definition} \label{def:identBlaschke}
	If  $ z_1^\lambda , \dots , z_n^\lambda $ are the solutions of \eqref{solutionBlaschkeB}, then we say that Blaschke product $b_n$ \emph{identifies} the set of points $\mathcal{Z}^\lambda=\{z_1^\lambda , \dots , z_n^\lambda \}$. 
\end{definition}

\subsection{Tool set 3: Matrices and numerical range} \label{sec:numrange}

\subsubsection{Class $\mathcal{S}_{n}$.} \label{sec:Sn}
\

A square complex matrix $\bm A\in \mathbb C^{n\times n}$ is a \emph{completely non-unitary contraction} if $\|\bm A\|\leq 1$ and all eigenvalues  of $\bm A$ are strictly inside the unit disk $\mathbb D$.  The space $\mathcal{S}_{n}$ is the set of completely non-unitary contractions in $\mathbb C^{n\times n}$ with defect index  
$$
\rank(\bm I_n-\bm A \bm A^{*})=\rank(\bm I_n-\bm A^{*}\bm A )=1.
$$
The spaces $\mathcal{S}_{n}$ and their infinite-dimensional analogues have been studied extensively, initially in the work of Livshitz \cite{MR0020719}, and in the 1960s by Sz.-Nagy and collaborators \cite{MR2760647}. 
A canonical example of a matrix in $\mathcal{S}_{n}$ ia a shift operator or $n\times n$ nilpotent Jordan block
\begin{equation}\label{eq:Jordan}
\bm J_n=\begin{pmatrix}
	0 & & &  & &\\
	1 & 0 & & & & \\
	& 1 & 0 & & \\
	& & \ddots & \ddots &  \\
	& & & 1 & 0
\end{pmatrix}_{n \times n}.
\end{equation}

A remarkable property of the class $\mathcal{S}_{n}$ is that the spectrum $\sigma(\bm A)$ of a matrix $\bm A\in \mathcal{S}_{n}$ determines the unitary equivalence class of $\bm A$. It turns out that  the equivalence class of $\bm A\in \mathcal{S}_{n}$ is determined even by  its numerical range $W(\bm A)$ (see the definition in Section~\ref{sec:numran} below).

There are several ``canonical'' representations for matrices (or rather, of their unitary equivalence classes) in $\mathcal{S}_{n}$, each having its own merit. For instance, we can think of elements of $\mathcal{S}_{n}$ as the ``compressions of the shift'' \cite{MR2760647} or ``compressed multiplication operators'' \cite{MR3945586} in a Hardy space setting.  
We can characterize $\bm A \in \mathcal{S}_{n}$ via their singular value decomposition (SVD), $\bm A=\bm U \bm D \bm V^{*}$, where $\bm U$ and $\bm V$ are unitary and $\bm D$ is the diagonal matrix $\diag(1,..., 1, a)$ with $0 \leq a <1 $ (see \cite{Wu:2000bn}).  Another representation, also in  \cite{Wu:2000bn}, says that each such a matrix $\bm A=\left(a_{i j}\right)_{i, j=1}^{n}\in \mathcal{S}_{n}$  if and only if it is unitarily similar to an upper triangular matrix with elements satisfying  $\left|a_{i i}\right|<1$ for all $i$, while for $i<j$,  
$$
a_{i j}=b_{i j}\left(1-\left|a_{i i}\right|^{2}\right)^{\frac{1}{2}}\left(1-\left|a_{j j}\right|^{2}\right)^{\frac{1}{2}}, \quad  b_{i j}= \begin{cases} 
\displaystyle	\prod_{k=i+1}^{j-1}  \left(-\bar{a}_{k k}\right) &  \text { if } i<j-1 \\[1mm]
	1 & \text { if } i=j-1.
\end{cases} 
$$

From our observations above it follows that the cut-off CMV matrix $\cmv{n}$ is in  the class $\mathcal{S}_{n}$.  In fact, from  \cite[Theorem 2]{MR3945586} we know that every matrix from the class $\mathcal{S}_{n}$ is unitarily equivalent to a cutoff CMV matrix.  In short, CMV matrices are another canonical representation of elements in $\mathcal{S}_{n}$, which has several advantages.
For instance, it  gives us an effective construction of the equivalence class of $\bm A\in \mathcal{S}_{n}$ from its eigenvalues $f_1, \dots, f_n$: from the monic polynomial $\Phi_n(z; f_1,\dots, f_n)$ use inverse Szeg\H{o} recursion to obtain the Verblunsky coefficients $\alpha_0, \dots, \alpha_{n-1}$ and take $\bm A=\cmv{n}(\alpha_0, \dots, \alpha_{n-1})$.

\subsubsection{Numerical range.}  \label{sec:numran}
\

The \emph{numerical range} or \emph{field of values} of  a matrix $\bm A\in \C^{n\times n}$ is the subset of the complex plane $\C$ given by
$$
W(\bm A)=\{ x^* \bm A x:\, x\in \C^n, \; x^* x=1 \}.
$$
In other words, the numerical range is the image of the Euclidean unit sphere under the continuous map $x \mapsto x^*\bm Ax$.

For a matrix $\bm A$, the numerical range $W(\bm A)$ is a compact and convex subset of $\C$ (Toeplitz--Hausdorff Theorem) that contains  the spectrum $\sigma(\bm A)$ of $\bm A$, and it is invariant by unitary conjugation of $\bm A$. This shows that for normal matrices we can reduce the analysis of $W(\bm A)$ to the case of diagonal $\bm A$; a straightforward consequence is that for a normal $\bm A$, $W(\bm A)$  is the convex hull of $\sigma(\bm A)$. In particular, for a unitary matrix $\bm A$, its numerical range is a convex polygon with vertices at its eigenvalues, inscribed in the unit circle $\mathbb T$. 

The so-called Elliptical Range Theorem (see e.g. \cite[Chapter 6]{MR3932079}, \cite[\S 1.3]{MR2978290} or \cite{MR1322932}) says that if $n=2$, then $W(\bm A)$ is an ellipse with eigenvalues $f_1$ and $f_2$ as foci, and the minor axis of length 
$$
\sqrt{\operatorname{tr} (\bm A^* \bm A) - |f_1|^2 - |f_2|^2}.
$$
The numerical range of the Jordan block \eqref{eq:Jordan} is the circular disk with center at $0$ and radius $r=\cos(\pi/(n+1))$ \cite[Proposition 1]{MR1072339} (notice that for $n=2$ it satisfies Chapple's formula \eqref{eq:Chapple}).

The most general statement in this direction is Kippenhahn's theorem, see Theorem~\ref{thm:Kippenhahn} in the Introduction, which states that  \emph{for $\bm A\in \C^{n \times n }$ there exists a real algebraic curve $\Gamma$ of class $n$ whose foci are  the eigenvalues of $\bm A$, such that $W(\bm A)$ is the convex hull of $\Gamma(\R)$.}  

In fact, the proof of Kippenhahn's theorem is constructive and contains the derivation of an equation for the dual $\Gamma^*$ of the algebraic curve $\Gamma$.  
Indeed, for the given matrix $\bm A\in \C^{n \times n }$,  consider the homogeneous  polynomial 
\begin{equation} \label{defKippenhahn}
	G_{\bm A}(u_1, u_2, u_3):=\det (u_1 \Re \bm A+u_2 \Im \bm A - u_3 \bm I),
\end{equation}
where $\Re \bm A:=\left(\bm A+\bm A^{*}\right) / 2$ and $\Im \bm A:=\left(\bm A-\bm A^{*}\right) /(2 i)$ are the real and imaginary parts of $\bm A$, respectively, and $\bm I $ in this case denotes the $n \times n $ identity matrix. It is easy to see that $G_{\bm A}(u_1, u_2, u_3)$ is a homogenous polynomial of degree $n$ with real coefficients.
Thus  $G_{\bm A}(u_1, u_2, u_3)=0$ defines a real algebraic curve $\Gamma_{\bm A}^*\subset \P^2(\C)$ of degree $n$ that is the dual of an algebraic curve $\Gamma_{\bm A} \subset \P^2(\C)$
of class $n$.  We will call $\Gamma_{\bm A}$ the \emph{Kippenhahn curve}  of $\bm A$.
As it follows from \cite{MR59242, MR2378310} (see also \cite[Theorem 6.1]{Langer:2007bv}), if we denote by $\lambda_\varphi\in \R$, $\varphi\in [0,2\pi]$, the maximal eigenvalue of the matrix $\Re(e^{-i\varphi} \bm A)$ then $(\cos \varphi: \sin \varphi: \lambda_\varphi)\in \P^{2}(\R)$ belongs to $\Gamma_{\bm A}^*(\R)$ and
the equation
$$
u_1 \cos \varphi + u_2 \sin \varphi  - u_3 \lambda_\varphi=0
$$
defines a supporting line to $W(\bm A)$. In consequence,  
the numerical range $W(\bm A)$ is the convex hull of $\Gamma_{\bm A}(\R)$. 
Finally, as seen in \eqref{eq:dualG},  the real foci of $\Gamma_{\bm A}$ are the solutions of
$$
G_{\bm A}(1, i,  z)=\det ( \bm A-\left. z \bm I \right)=0,
$$
that is, the eigenvalues of $\bm A$. 

\subsubsection{Dilations.} \label{sec:dilations}
\

We say that an $m \times m$ matrix $\bm A$ \emph{dilates} to the $n\times n$  matrix $\bm B$ ($m<n$) if there is an isometry $\bm V$ from $\mathbb{C}^{m}$ to $\mathbb{C}^{n}$ such that $\bm A=\bm V^{*} \bm B \bm V$. This is equivalent to saying that $\bm B$ is unitarily similar to an $n\times n$ matrix of the form $$
\begin{pmatrix} \bm A & * \\ * & *\end{pmatrix} ,
$$ 
in which $\bm A$ appears in the upper left corner. It is a well-known fact that the numerical range is monotone by dilation:   if $\bm A$ dilates to $\bm B$ then  $W(\bm A)\subset  W(\bm B)$.

An important dilation, which is also closely related to the results we discuss, is the \emph{unitary dilation of completely non-unitary contractions}. The notion of completely non-unitary contraction was already introduced in Section~\ref{sec:numrange}.  A classical result of Halmos \cite[Problem 222(a)]{Halmos82} says that every completely non-unitary contraction $\bm A$ has unitary dilations.

As we have seen in Section~\ref{sec:numrange}, OPUC provides not only an effective construction  of such a matrix $\bm A$ from any equivalence class of $\mathcal{S}_{n-1}$, but also of its $1$-parametric family of unitary dilations. Indeed, for $\bm A=\cmv{n-1}(\alpha_0, \dots, \alpha_{n-2}) \in  \mathcal{S}_{n-1}$, its unitary dilations are given by  $\cmv{n}(\alpha_0, \dots, \alpha_{n-2}, \lambda)$, with $\lambda \in \T$.

\section{Curves inscribed in polygons} \label{sec:evolvents}

\subsection{Definitions} \label{sec:defs} 
\

Poncelet's closure theorem  (see Theorem~\ref{thm:Poncelet} in the Introduction) sets up a  leitmotiv of this paper: curves in $\mathbb D$ that are tangent to polygons  with vertices on $\mathbb T$. More precisely, we will be interested in real algebraic curves that satisfy the Poncelet porism: if it is inscribed in a certain polygon, it is so for the whole continuum of such polygons (in other words, curves that are envelops of such a family of polygons). We need to make all these notions rigorous.

\subsubsection{Polygons} \label{def: polygonal chain}
\

We denote by $[a,b]$ the straight segment joining points $a, b\in \C$.

For $n\geq 3$, a \emph{polygon} with $n$ vertices  in $\C$ is a union of $n$ distinct segments  
\begin{equation}\label{eq:polygonal chain}
\mathscr{P}=[z_1,z_2]\cup [z_2,z_3] \cup \cdots  [z_{n-1},z_{n}] \cup [z_{n},z_{1}],
\end{equation}
where $z_1,\ldots,z_{n}\in \C$ are $n$ distinct points, such that any two segments intersect in at most one point and 
 any two segments sharing a common endpoint are noncollinear.
The $n$ segments are called the \emph{sides} or the \emph{edges} of the polygon.
If the sides of $\mathscr{P}$ only intersect when they share a common endpoint, then $\mathscr{P}$ is called   \emph{simple}, and if all its vertices belong to $\T$ we say that $\mathscr{P}$ is \textit{inscribed} in $\T$. 
It is well-known that any simple polygon inscribed in $\bbT$ is convex.

We will often view the segments  in (\ref{eq:polygonal chain}) as directed line segments
and hence $\mathscr{P}$  as an oriented  piecewise linear closed curve; in consequence, the notation $-\mathscr{P}$   stands for the polygon $\mathscr{P}$  traversed in the opposite direction. 
Note that every simple polygon is a Jordan curve and hence
divides the complex plane $\C$ into an interior region and an exterior region. As usual,   we say that a simple   polygon  $\mathscr{P}$ is 
\textit{positively oriented} if the interior is to the left. Equivalently, a simple   polygon  is 
positively oriented if its winding number with respect to any point in its interior  is equal to $1$.  
For a general (not necessarily simple) oriented  polygon   $\mathscr{P}$, if  we traverse $\mathscr{P}$ in the direction of the orientation, 
then at each vertex we turn by a nonzero angle between $-\pi$ and $\pi$. The sum of these turning angles divided by $2\pi$ is called the \emph{turning number} of 
$\mathscr{P}$.  If $\mathscr{P}$ is simple, then the turning number is equal to  the winding number with respect to a point in the interior.
We say that $\mathscr{P}$ is positively oriented if its turning number is positive.

It will be  convenient to also consider ``degenerate'' polygons with two  vertices.
These would be chains of the form $\mathscr{P}=[z_1,z_2]\cup [z_2,z_1]$, where we view $[z_1,z_2]$ and $[z_2,z_1]$
as distinct directed segments. By convention, the turning number of such a polygon is equal to $\pm 1$.

\subsubsection{Poncelet curves and Poncelet correspondence}   \label{sec:Ponceletcorresp}
\

\begin{definition}\label{def:closedcurve}
By a \textit{closed curve}  we mean a subset  $C \subset \P^2(\R)$ such that there exist a continuous map $\gamma:\bbT\rightarrow  \P^2(\R)$ with $C =\gamma(\bbT)$ and a real algebraic curve $\Gamma$ such that $C\subset \Gamma(\R)$ with the additional condition that either all points of $C$ are nonsingular or the parametrization $\gamma$ can be chosen such that if $\gamma(t_0)$ is a singular point of $\Gamma$, then for some sufficiently small interval $I\subset \bbT$ containing $t_0$, the curve $\gamma(I)$ lies in a single local branch of $\Gamma$.   
If $C$ is nonsingular, then $C$ is a simple closed curve in $\C$. 
We also allow for the degenerate case when $C$ consists of  just a single point.
\end{definition}
Notice that algebraicity is built into our definition of a closed curve; in the following, all closed curves that we consider will be duals of smooth closed curves in $ \P^2(\R)$.

\begin{definition}  \label{defPonc}
For $n\geq 2$, we say that a set of (not necessarily simple) $n$-sided polygons $\{ \mathscr{P}(z):\, z\in \bbT\}$  inscribed in $\bbT$ is a \emph{family of $n$-Poncelet polygons} if for each $z\in \bbT$, $z$ is one of the vertices of $\mathscr{P}(z)$, and the following condition holds:
$$
\text{ $w\in \bbT$ is a vertex of $\mathscr{P}(z)$} \quad \Rightarrow \quad \mathscr{P}(z)= \mathscr{P}(w).
$$
A closed curve $C \subset \D$ is a \emph{Poncelet curve of rank $n$} or an \emph{$n$-Poncelet curve} with respect to $\bbT$
if  there is a family of $n$-Poncelet polygons $\{ \mathscr{P}(z):\, z\in \bbT\}$  such that 
\begin{enumerate}[i)]
\item for every $z\in \bbT$, each  of the $n$ sides of $\mathscr{P}(z)$ is tangent to $C$, and each  tangent line  of $C$ passing through $z$
contains a side of $\mathscr{P}(z)$;
\item for every $z\in \bbT$, the two sides of $\mathscr{P}(z)$ with the common vertex at $z$ are the only tangents to $C$ emanating from $z$;
\item for every $\zeta\in C$ there exists $z\in \bbT$ such that one of the sides of $\mathscr{P}(z)$ is tangent to $C$ at the point $\zeta$.
\end{enumerate}
\end{definition}
Notice that an $n$-Poncelet curve $C$ determines the family $\{ \mathscr{P}(z):\, z\in \bbT\}$ uniquely. Observe also that we set a convention that $C\subset \bbD$ is a necessary condition for $C$ being called a Poncelet curve.

\begin{remark}
\begin{enumerate}[i)]
\item If $\mathcal{Z}^\lambda=\{z_1^\lambda , \dots , z_n^\lambda \}$ is a one-parametric family of  pair-wise distinct points on $\bbT$ such that for every $z\in \bbT$ there exists a unique value of the parameter $\lambda$ for which $z\in \mathcal{Z}^\lambda$, then convex hulls of $\mathcal{Z}^\lambda$  constitute a family of convex $n$-Poncelet polygons. An example of such a construction is the sets of points identified by a Blaschke product, see Definition~\ref{def:identBlaschke}. 

\item Clearly, convexity of a Poncelet curve does not imply convexity of its Poncelet polygons, see for example curve $C_2$ in Figure~\ref{fig:convexity}, left. Surprisingly, the converse does not hold either (despite an assertion in \cite[Remark 1]{Mirman:2003ho}): there are non-convex Poncelet curves with the corresponding family of convex Poncelet polygons; see Figure~\ref{fig:convexity}, right, for a non-convex $3$-Poncelet curve. Its construction is explained in Example~\ref{exampleMirman} below.

\begin{figure}[htb] 
	\begin{center}
		\begin{tabular}{c p{3mm}c}		
			\begin{overpic}[width=0.35\linewidth]{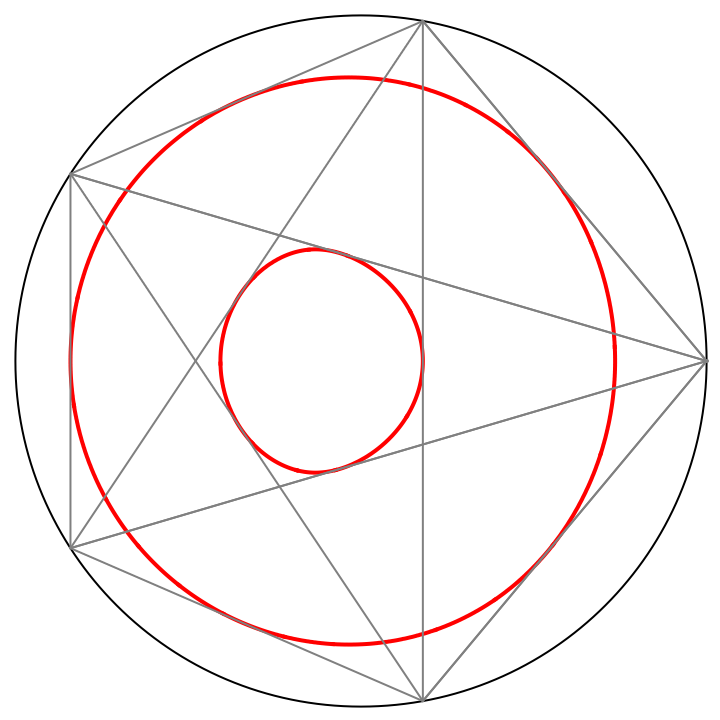} 
				\put(47,120){\small $C_1$}
				\put(50,67){\small $C_2$}	
			\end{overpic}
				& 	& \begin{overpic}[width=0.35\linewidth]{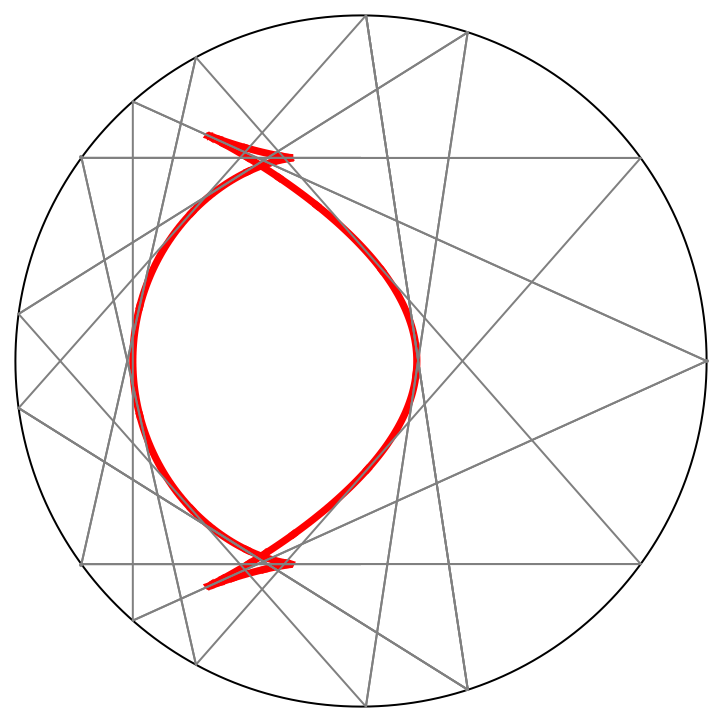} 
					\put(33,85){\small $C$}
				\end{overpic}
		\end{tabular}
	\end{center}
	\caption{Left: two convex $5$-Poncelet curves: for $C_1$ all its Poncelet polygons $\mathscr{P}(z)$ are convex, while for $C_2$ they are not. Right: a non-convex $3$-Poncelet curve.}
	\label{fig:convexity}
\end{figure} 
\end{enumerate}
\end{remark}

Given a family of positively oriented $n$-Poncelet polygons $\{ \mathscr{P}(z):\, z\in \bbT\}$ in the sense of Definition~\ref{defPonc}, we define the map $\tau: \T \rightarrow \T$ as follows: if $[z,w]$ is the edge of the polygon $\mathscr{P}(z)$ emanating from $z\in \bbT$ when traversed in the positive direction, then $\tau(z)=w\in \bbT$. The sequence $\{\tau^j(z)\}_{j=1}^\infty$ is periodic with period $n$ (that is, $n$ is the smallest positive integer such that $\tau^n = \text{Id}$, where $\text{Id}$   is the identity operator), and the positively oriented $n$-Poncelet polygons $\{ \mathscr{P}(z):\, z\in \bbT\}$ can be written as
\begin{equation} \label{eq: Ponclet polygons}
\mathscr{P}(z) = [z,\tau(z)]\cup  [\tau(z),\tau^2(z)] \cup \cdots \cup  [\tau^{n-1}(z),z].
\end{equation}

The  map $\tau: \T \rightarrow \T$, also known as the \textit{Poncelet correspondence}, see \cite{MR2465164} (and  in some particular cases is related to the John mapping, see \cite{Burskii:2013cg}), provides a connection between Poncelet curves and discrete dynamical systems and ellipsoidal billiards, see e.g. the work of Dragović and collaborators \cite{Dragovic:2019iy, Andrews:2019ty} and Schwarz \cite{Ovsienko:2010kj, Schwartz:2015cn}.  Notice that in the construction of $\tau$ we could start alternatively from an $n$-Poncelet curve $C$ and use its family of $n$-Poncelet polygons $\{ \mathscr{P}(z):\, z\in \bbT\}$ to define $\tau$. We develop this idea further in the next section, with the notion of associated Poncelet curves.

\subsubsection{Associated Poncelet curves} \label{def:assocPoncelet} 
\

Let $\{ \mathscr{P}(z):\, z\in \bbT\}$ be a family of $n$-Poncelet polygons  and $\tau: \T \rightarrow \T$ the Poncelet correspondence defined above. We will also assume that  $\tau$ is smooth  and 
\begin{equation}\label{eq: monotonicity assumption}
 \frac{d}{d\theta} \operatorname{arg} \tau(e^{i\theta})>0\ \mbox{for all $\theta \in \R$},
\end{equation}
where  the argument of $\tau(e^{i\theta})$ is viewed as a smooth  $\R$-valued function of $\theta \in \R$. For instance, the condition~\eqref{eq: monotonicity assumption} automatically holds if  the family $\{ \mathscr{P}(z):\, z\in \bbT\}$  corresponds to a convex and nonsingular $n$-Poncelet curve $C\subset \D$.

For $k\in \N$, we  define a curve $C_k \subset \bbD$ as the envelope of all   chords $[z ,\tau^k(z)]$, where $z\in \bbT$. A priori, it is not evident that 
this definition makes sense. For example, we really should consider the envelope of all  lines  determined by the chords $[z ,\tau^k(z)]$ since 
it is not clear that the points of tangency must lie on these chords. This is rather subtle and we will see later (see Theorem~\ref{lemma:inside}) that it is precisely the monotonicity condition \eqref{eq: monotonicity assumption} which  implies this.
 
We can make this more precise (and constructive) as follows. Define a map $\zeta_k: \bbT \rightarrow \P^2(\R)$ by
\begin{equation}
\zeta_k(z):= \frac{2 z \,	\tau^{k}(z)}{z+ 	\tau^{k}(z)} =
\mbox{pole of the line containing $[z,	\tau^{k}(z)]$,}  
\end{equation}
see \eqref{correspUandVb}. By construction, if $z\neq \tau^k(z)$ then $\zeta_k(z)$ lies outside of the unit disk. Since $\tau^n = \text{Id}$, we only need to consider $1\leq k\leq n-1$.  Notice that $\zeta_2$ is related to the pentagram map, see e.g.~\cite{Schwartz:1992, Schwartz:2015cn}.

\begin{lem}\label{lemma:immersion}
	For each $1\leq k\leq n-1$, the map $\zeta_k: \bbT \rightarrow \P^2(\R)$  is a smooth immersion, and the image $\zeta_k(\bbT)$ is   nonsingular.
\end{lem}
\begin{proof}
 By our assumptions,  the map $\tau:\T \rightarrow \T$ is smooth so that each of the maps
 $\zeta_k: \bbT \rightarrow \P^2(\R)$  is also smooth, and the argument of $\tau(e^{i\theta})$ 
 is a strictly increasing smooth function  when viewed as a continuous $\R$-valued function of $\theta \in \R$. 
 It then also follows that the argument of $\tau^k(e^{i\theta})$ is a strictly increasing   function of $\theta \in \R$. 
 As a consequence, the differential of the map $\zeta_k: \bbT \rightarrow \P^2(\R)$ is  nonzero everywhere
 and hence $\zeta_k$ is a smooth immersion.  Since an immersion is a local embedding, it follows that $\zeta_k(\bbT)$ is  nonsingular.
\end{proof}

\begin{definition}
For each $1\leq k\leq n-1$, the curve  $C_k$ is the dual  of the nonsingular curve $\zeta_k(\bbT)$.  
\end{definition}
Notice  that $C_k$ may have  singularities. As usual, cusps of $C_k$ correspond to tangent lines at inflection points of  $\zeta_k(\bbT)$
and double points of $C_k$ correspond to  lines that are tangent at two distinct points of  $\zeta_k(\bbT)$, see e.g.\ Figure~\ref{fig:Dual}.

\begin{thm}\label{lemma:inside}
For each $1\leq k\leq n-1$, assumption \eqref{eq: monotonicity assumption} implies that $C_k\subset \D$. 
If 
$$
\frac{d}{d\theta} \operatorname{arg} \tau(e^{i\theta})\geq 0\ \mbox{for all $\theta \in \R$}
$$
then $C_k\subset \overline{\D}$ and can have points on $\T$ (see Figure~\ref{fig:cusponT}). 

Finally, if  at some $\theta \in \R$, 
$$
\frac{d}{d\theta} \operatorname{arg} \tau(e^{i\theta})< 0
$$
then $C_k$ has points outside $ \overline{\D}$.
\end{thm}
\begin{proof}
Obviously, it is sufficient to prove the statement for $k=1$. 

The curve  $C_1$ is obtained from the tangent lines to  curve $\zeta_1(\bbT)$ via reciprocation in $\bbT$. In particular, $C_1$ contains a point outside of $\overline{\D}$ if and only if there is a tangent line to $\zeta_1(\bbT)$  that intersects $\T$ in two distinct points (or alternatively, whose distance to the origin is $<1$).  

In order to simplify notation, denote
$$
w=\tau(z), \quad w' = \frac{d}{dz}  \tau(z), \quad \dot w = \frac{d}{d\theta} \operatorname{arg} \tau(e^{i\theta}).
$$
Since $w\in \bbT$, $\log w = i\arg w$, we get that
$$
\dot w = \frac{z w'}{w}.
$$
Thus, differentiating $\zeta_1(z)=  2 z 	w /(z+w)$ we get that
$$
  \frac{d}{d\theta} \zeta_1  = i z \, \zeta_1'(z)= i  \zeta_1 \,  \frac{w + \dot w z}{w+z} , \quad z=e^{i\theta}, \quad \zeta_1=\zeta_1\left(z \right).
$$
Hence, the parametric equation of the straight line tangent to $\zeta_1(\bbT)$ at the point $\zeta_1(e^{i\theta})$ is 
$$
L(t)= \zeta_1 + it \,  \zeta_1 \,  \frac{w + \dot w z}{w+z} =  \frac{2 z 	w}{z+w} \left(1 + it \,   \frac{w + \dot w z}{w+z}   \right) , \quad t\in \R.
$$
Notice that $\dot w\in \R$, $z, w\in \bbT$, so that 
$$
\overline{L(t)}=  \frac{2  }{z+w} \left(1 - it \,   \frac{z + \dot w w}{w+z}   \right) ,
$$
and hence,
\begin{align*}
|L(t)|^2 & =  \frac{4 z 	w}{(z+w)^2} \left(1 + it \,   \frac{w + \dot w z}{w+z}   \right) \left(1 - it \,   \frac{z + \dot w w}{w+z}   \right)  = \frac{4 w z }{(w+z)^4} \,  ( \alpha t^2 + \beta t +\gamma ), 
\end{align*}
with
$$
\alpha = (w + \dot w z)(z + \dot w w) , \quad \beta = i (w^2-z^2)(1-\dot w), \quad \gamma = (z+w)^2.
$$
This is a quadratic function in $t$ whose minimum is attained at $t=-\beta/(2\alpha)$. Replacing it in the expression for $|L(t)|^2$ we get that the square of the distance of the tangent line to the origin is 
$$
\frac{z w (1+ \dot w)^2  }{ (w + \dot w z)(z + \dot w w) } = \frac{ (1+ \dot w)^2  }{  |z + \dot w w|^2  }. 
$$
If $\dot w\geq  0$ then by the triangle inequality, 
$$
 |z + \dot w w|\leq  |z| + |\dot w w|\ = 1 + \dot w,
$$
which shows that the distance is $\ge 1$. Moreover, equality holds only when either $z=w$  or when $\dot w=0$. In the same vein, the reversed triangle inequality and the assumption $\dot w<0$ yields that the distance is $<1$, so that the tangent line intersects the unit circle in two points.
\end{proof}

\begin{figure}[htb] 
	\begin{center}
		\begin{tabular}{c c c}		
			\includegraphics[width=0.25\linewidth]{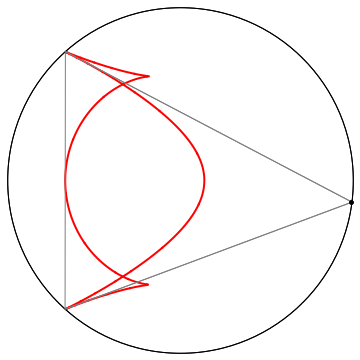} & \includegraphics[width=0.25\linewidth]{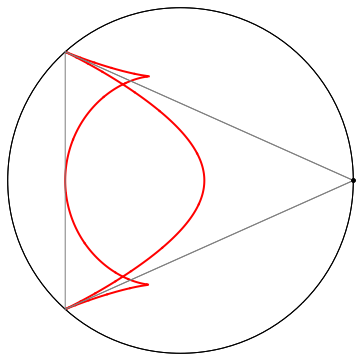} 	& \includegraphics[width=0.25\linewidth]{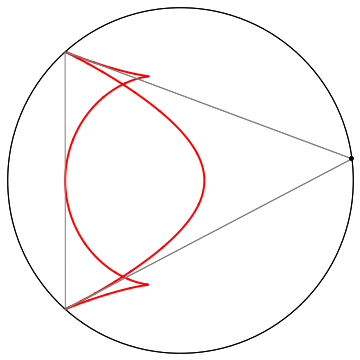} 
		\end{tabular}
	\end{center}
	\caption{A family of $3$-Poncelet polygons such that the curve $C_1$ has a cusp on $\T$. The highlighted point on $\T$ 
	corresponds to $e^{i\theta}$ for the values $\theta=-\pi/25$, $0$, and $\pi/25$, respectively. The function $\operatorname{arg} \tau(e^{i\theta})$ is strictly increasing but has a stationary point at $\theta=0$; $C_1$ intersects $\bbT$.}
	\label{fig:cusponT}
\end{figure}

\begin{lem}
For each $1\leq k\leq n-1$, the map $\zeta_k: \bbT \rightarrow \P^2(\R)$  is one-to-one, unless  $n=2k$, in which case  $\zeta_k$ is two-to-one. Additionally, $\zeta_k(\bbT)=\zeta_{n-k}(\bbT)$.
\end{lem}

\begin{proof}
The case $n=2$ is trivial, so assume $n\ge 3$. 
Let $z,w\in \bbT$ such that $\zeta_k(z)=\zeta_k(w)$. Then 
$[z,\tau^{k}(z)]=[w,\tau^{k}(w)]$ and hence either $z=w$ or $\tau^k(z)=w$ and $z=\tau^k(w)$.
The second condition  is equivalent to $\tau^{2k}(z)=z$ which is only possible if $n=2k$.
It follows that the map $\zeta_k$ is one-to-one if $n\not=2k$ and two-to-one if $n=2k$.

The last assertion follows from the fact that if $z\in \T$ and $w=\tau^k(z)$, then by definition $\zeta_k(z)$ is the pole of the line containing $[z,w]$. At the same time,
$\tau^{n-k}(w)=\tau^n(z)=z$, so that $\zeta_{n-k}(w)$ is the pole of the line containing $[w,z]$, which shows that $\zeta_{n-k}(w)=\zeta_k(z)$, with $w=\tau^k(z)$.
\end{proof}

\begin{figure}[htb] 
	\begin{center}
			\begin{overpic}[width=35pc]{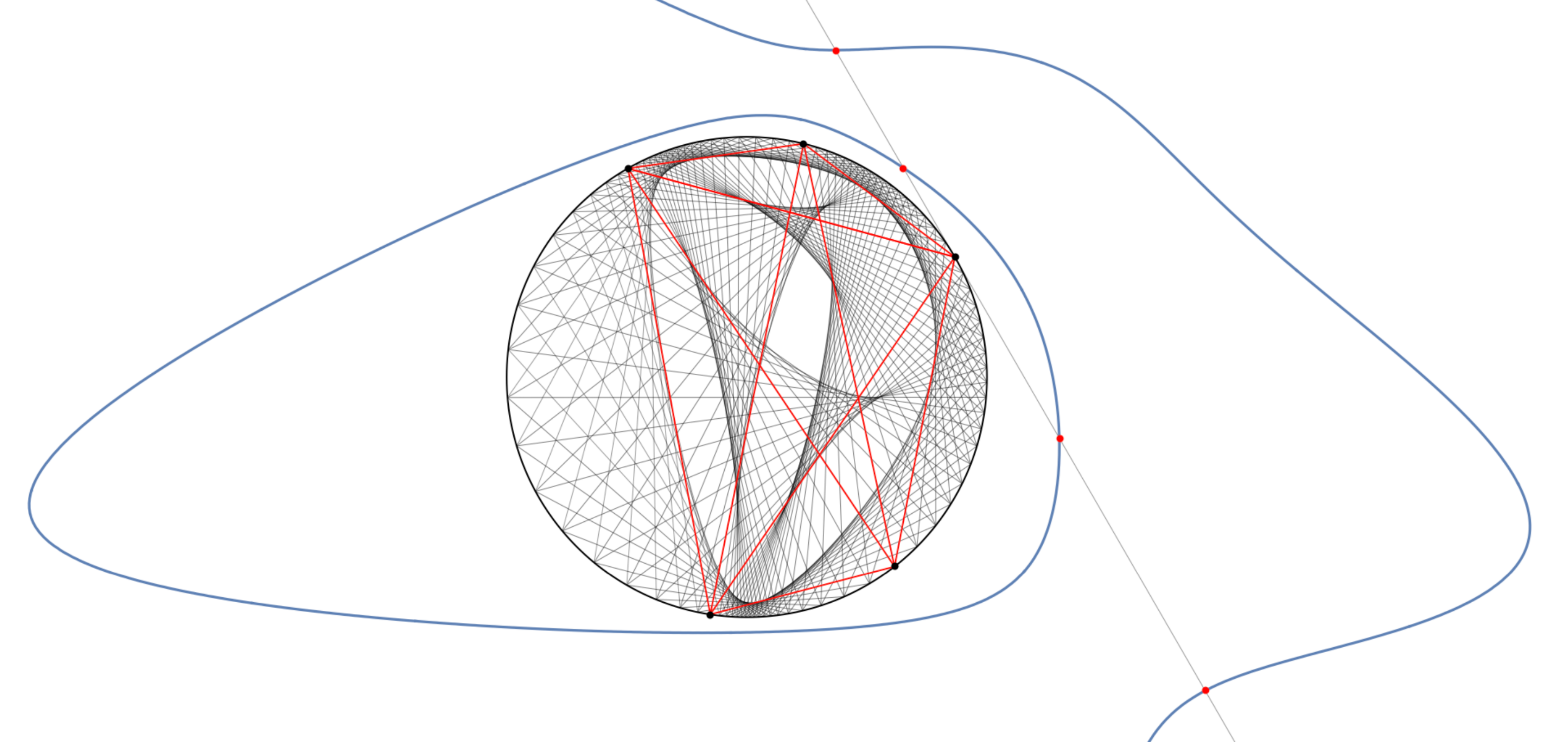}
			\put(70,20){$C_1^*$}
				\put(167,70){$C_1$}
			\put(360,125){$C_2^*$}
				\put(207,105){$C_2$}
			\put(120,85){ $\bbT$}
		\end{overpic}
	\end{center}
	\caption{The curves $C_1^*$ and $C_2^*$ associated to a convex Poncelet curve of rank $5$.}
	\label{fig:Dual}
\end{figure}
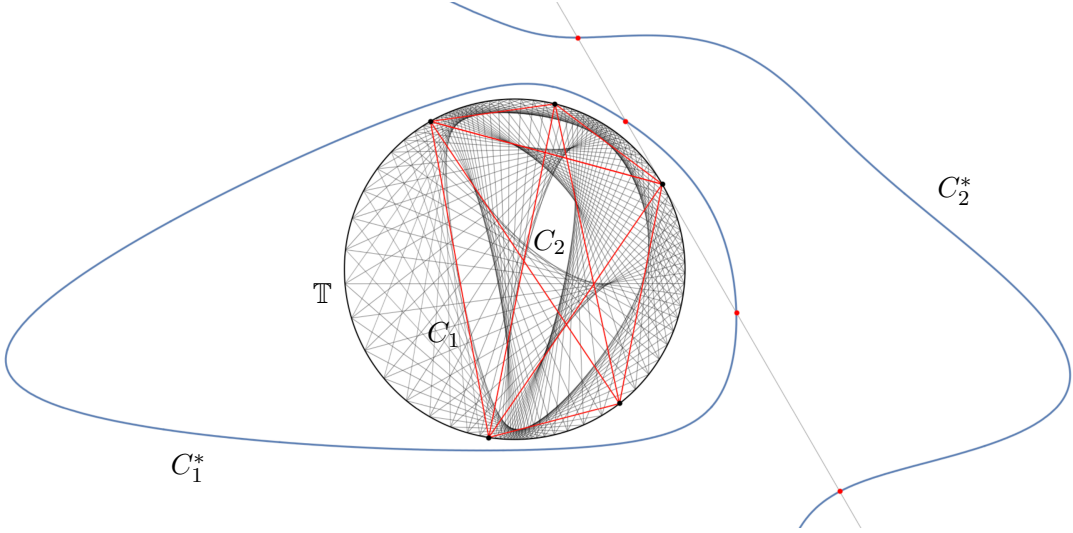 

This lemma shows that $C_{n-k}=-C_k$, and we only need to consider curves $C_k$ for $1\leq k\leq [n/2]$. It turns out that each such a component $C_k$  exhibits the Poncelet property:
 \begin{thm} \label{thm:Gauss}
Assume  that $n\geq 3$ and all $C_k$, $1\leq k\leq [n/2]$, are closed curves in the sense of Definition~\ref{def:closedcurve}. Then,  for  each $1\leq k\leq [n/2]$, the curve $C_k \subset \bbD$  is a Poncelet curve of rank  $n/\gcd(k,n)$. Moreover, if all Poncelet polygons $\mathscr{P}(z)$  are convex then  the  positively oriented Poncelet polygons  for $C_k$ have turning number $k/\gcd(k,n)$.
 
Furthermore, if $d$ is a divisor of  $n$ and $d\ge 3$, then the number of curves $C_k$, $1\leq k \le [n/2]$, that have the $d$-Poncelet property  is $\phi(d)/2$, where $\phi$ denotes Euler's totient function  (i.e., $\phi(d)$ counts the positive integers up to $d$ that are relatively prime to $d$). 
\end{thm}
Recall that here we assume that condition \eqref{eq: monotonicity assumption} holds.
\begin{proof}

For the following, set  $\Z_n:=\{0,1,2,\ldots, n-1\}$ and view $\Z_n$ as an additive group, where addition is the usual addition of integers modulo $n$.
It is an elementary result in basic group theory that the order of $k\in \Z_n$ is equal to $n_k:=n/\gcd(k,n)$.
Here the order of   $k\in \Z_n$ is defined as the smallest positive integer $m$
such that
$$
mk:=\underbrace{k+\cdots+k}_{m}=0\ \mbox{in $\Z_n$}.
$$
For $ k\in \N$ and $z\in \bbT$,  define a polygon with $n_k$ sides by 
\begin{equation} \label{defPoncPol}
\mathscr{P}_k(z) :=  [z,\tau^k (z)]\cup  [\tau^k(z),\tau^{2k}(z)] \cup \cdots \cup  [\tau^{(n_k-1)k}(z), z].
\end{equation}
Since $\tau^n(z)=z$ we can view the exponents in the squence $z$, $\tau^k(z)$, $\tau^{2k}(z)$, etc., as elements of $\Z$ or as elements of $\Z_n$.
 
Note that each one of the segments is of the form $[w,\tau^k(w)]$ for some $w\in \bbT$. In fact, for any $0\leq m\leq n_k-1$, we have
$$
[\tau^{mk}(z),\tau^{(m+1)k}(z)] = [w,\tau^k(w)],\  \text{where $w=\tau^{mk}(z)$}.
$$ 
Thus, $C_k$ is an $n_k$-Poncelet curve with Poncelet polygons $\mathscr{P}_k(z)$.

Moreover, if all Poncelet polygons $\mathscr{P}(z)$  are convex, then for  each $z\in \bbT$, the vertices $z,\tau(z),\ldots , \tau^{n-1}(z)$ of $\mathscr{P}(z)$ are points on $\bbT$ in counterclockwise order.
Since $n_k=n/\gcd(k,n)=\min\{m \in \N :  \mbox{$mk$ is a positive multiple of $n$}\}$ and
$$
\frac{n}{\gcd(k,n)}\, k = \frac{k}{\gcd(k,n)}\, n,
$$
it follows that the turning number of $\mathscr{P}_k(z)$ is $k/\gcd(k,n)$.

A consequence of the just established fact is that a polygon $\mathscr{P}_k(z)$ in \eqref{defPoncPol}
has exactly $n$ segments if and only if $\gcd(n,k)=1$. Thus, the total number of such polygons is precisely the number of integers $1\leq k\leq [n/2]$ with $\gcd(n,k)=1$, which coincides with $\phi(n)/2$. More generally,  suppose $d$ is a divisor of $n$,  $d\ge 2$. A well known fact is that  the number of integers $1\leq k\leq [n/2]$ with $\gcd(n,k)=n/d$ is equal to $\phi(d)/2$, which proves the theorem. 
\end{proof}

\subsubsection{Complete Poncelet curves} \label{def: CompleteP}
\

Let $\{ \mathscr{P}(z):\, z\in \bbT\}$ be a family of \textit{convex} $n$-Poncelet polygons  such that the Poncelet correspondence  $\tau: \T \rightarrow \T$ is smooth and  condition \eqref{eq: monotonicity assumption} holds.
Consistent with the hypotheses of Theorem~\ref{thm:Gauss}, we assume that each of the associated Poncelet curves $C_k$,  $1\leq k\leq  [n/2]$, as constructed in Section~\ref{def:assocPoncelet}, is closed (and thus, algebraic) in the sense of Definition~\ref{def:closedcurve}. Moreover, by Theorem \ref{lemma:inside}, all $C_k\subset \bbD$.

\begin{definition}\label{def:Package}
Under the assumptions above, the union
\begin{equation}\label{eq:package}
\mathcal K_n := \bigcup_{k=1}^{[n/2]} C_k  \subset \bbD
\end{equation}
is  called a \textit{package of Poncelet curves generated by} the family $\{ \mathscr{P}(z):\, z\in \bbT\}$. If $\{ \mathscr{P}(z):\, z\in \bbT\}$ are the convex Poncelet polygons for a closed curve $C\subset \bbD$, we alternatively say that the package of Poncelet curves is \textit{generated by $C$}; in this case, $C_1=C$. 
\end{definition}
This terminology was apparently introduced by Mirman, see  \cite{Mirman:2003ho}. 

Recall that by assumption, every package of Poncelet curves is algebraic, so that there exists a real real algebraic curve $\Gamma$ such $\mathcal K_n\subset \Gamma(\R)$.
\begin{lem} \label{lem:class}
Let $\Gamma$ be any real algebraic curve  such that  
\begin{equation}\label{eq: complete Poncelet}
\mathcal K_n=\bigcup_{k=1}^{[n/2]} C_k  \subseteq \Gamma(\R).
\end{equation}
Then the class of $\Gamma$ is at least $n-1$. 

If the class of $\Gamma$ is exactly $n-1$, then  (\ref{eq: complete Poncelet}) becomes
\begin{equation}\label{Cunion}
\mathcal K_n=	\bigcup_{k=1}^{[n/2]} C_k  = \Gamma(\R).
\end{equation}
\end{lem}
%
\begin{proof}
	First note  that every tangent line of one of the curves $C_k$ is also a  tangent line of $\Gamma(\R)$. Since $C_k=-C_{n-k}$ it then follows that $C_k^* \subset  \Gamma^*(\R)$ for all $1\leq k \leq n-1$.
	
	For every $z\in \bbT$, the line in $\P^2(\R)$ that is tangent to $\bbT$ at $z$ intersects  
	\[
	\mathcal{K}_n^*:= \bigcup_{k=1}^{[n/2]} C_k^*
	\]
	in  exactly $n-1$ distinct points,  namely the polars of the lines containing the diagonals $[z,\tau^k(z)]$, $k=1,\ldots,n-1$.
	Thus, since  $\mathcal{K}_n^* \subseteq \Gamma^*(\R)$, it follows by B\'ezout's theorem that the degree of $\Gamma$ must be at least $n-1$.

	Now suppose that the class of $\Gamma$ is exactly $n-1$.   By the argument above, if $l$ is any line  in $\P^2(\R)$ that is tangent to $\bbT$, then  $\mathcal{K}_n^* \cap l = \Gamma^*(\R)\cap l$.  Suppose that $\mathcal{K}_n \not= \Gamma(\R)$ or, equivalently, $\mathcal{K}_n \not= \Gamma(\R)$.
	Let $p \in \Gamma^*(\R) \setminus \mathcal{K}_n^*$ and let  $l$ be a line containing $p$ that is tangent to $\bbT$. Then for this line $l$ we would have  $\mathcal{K}_n^* \cap l \not= \Gamma^*(\R)\cap l$ which is a contradiction. Thus, if the class of $\Gamma$ is exactly $n-1$, then $\mathcal{K}_n = \Gamma(\R)$.
\end{proof}

If \eqref{Cunion} holds, then $\Gamma(\R)$ is a complete Poncelet curve in the sense of of the following definition.
\begin{definition} \label{def:completeNponcelet} 
If there exists a real algebraic curve $\Gamma$ such that \eqref{Cunion} holds	then the set of real points $\mathcal K_n=\Gamma(\R)\subset \bbD$ of a plane algebraic curve $\Gamma$ is called a 
 \emph{complete Poncelet curve} (also, a complete Poncelet--Darboux curve) \emph{of rank $n$} or a  \emph{complete $n$-Poncelet curve}
with respect to $\bbT$, generated by the family $\{ \mathscr{P}(z):\, z\in \bbT\}$.  

The dual $\Gamma^*$   is called the \emph{Darboux curve} for $\Gamma$. 
\end{definition}
If $\{ \mathscr{P}(z):\, z\in \bbT\}$ are the convex Poncelet polygons for a closed curve $C\subset \bbD$, we alternatively say that the complete $n$-Poncelet curve is \textit{generated by $C$}; in this case, $C_1=C$. 

Notice that if $\Gamma(\R)$ is a complete Poncelet curve of rank $n$ then  for every $z\in \bbT$  there exists an $n$-sided  polygon $\mathscr{P}(z)$ inscribed in $\bbT$ such that $z$ is one of these vertices,  each  of the $n(n-1)/2$ lines connecting the $n$ vertices of $\mathscr{P}(z)$ is tangent to $\Gamma(\R)$, and each  tangent line  of $\Gamma(\R)$  containing one of the vertices of the polygon must contain a side of $\mathscr{P}(z)$. Since in this construction we can always replace $\mathscr{P}(z)$ by the boundary of its convex hull, the convexity assumption on the family $\{ \mathscr{P}(z):\, z\in \bbT\}$  of $n$-Poncelet polygons is actually not a restriction and is made for convenience. Any other choice of polygons would simply imply a different enumeration of the components $C_k$.

\begin{example}
	For $n=24$, if $\Gamma(\R)$ is a complete Poncelet curve, then 
	$$
	\Gamma(\R)=\bigcup_{k=1}^{12} C_k,
	$$ 
	where $C_1$, $C_5$, $C_7$, and $C_{11}$ are $24$-Poncelet curves, $C_{2}$ and $C_{10}$ are $12$-Poncelet curves, $C_{3}$ and $C_9$ are $8$-Poncelet curves, $C_4$ is a $6$-Poncelet curve,  $C_6$ is a $4$-Poncelet curve, $C_8$ is a $3$-Poncelet curve, and $C_{12}$ is a $2$-Poncelet curve. Note that $C_{12}$ may consist of a single point, namely in the case when connecting opposite vertices of the inscribed dodecagons all meet in a single point.
\end{example}

\subsection{Mirman's parametrization of a package of Poncelet curves} \label{sec:MirmanPoncelet}
\

\subsubsection{From tangent coordinates to the Bezoutian form} \label{sec:Bezoutian}
\

Given a package of Poncelet curves $\mathcal K_n$ generated by a family of Poncelet polygons $\{ \mathscr{P}(z):\, z\in \bbT\}$ (or by a closed curve $C$), we can use Mirman's parametrization explained in Section~\ref{sec:MirmansParam} to derive an equation for the algebraic curve $\Gamma$ in the right hand side of \eqref{eq: complete Poncelet}--\eqref{Cunion}. Recall that if  $z=z_1\in \bbT$, $w=z_2\in \bbT$ are endpoints of a line tangent to $\Gamma(\R)$, then \eqref{MirmanSubstitution} yields  an equation of the form
\begin{equation} \label{eq:P=0}
P(z,w)=0,
\end{equation}
where $P$ is a polynomial, symmetric in $z$ and $w$. 

A crucial consequence of Definition~\ref{defPonc} is that  for every $w_0\in \bbT$ there exist \textit{exactly} $n-1$ solutions  $w_1, \dots, w_{n-1}$ of the equation 
\begin{equation} \label{eq:P1=0}
P(w_0, w)=0,
\end{equation}
namely the other vertices of the Poncelet polygon $\mathscr{P}(w_0)$, all of them  on $\bbT$, and that they satisfy that
\begin{equation} \label{eq:P2=0}
P(w_i, w_j)=0, \quad 0\le i < j \leq n-1.
\end{equation}
This shows that $P(z,w)$ is of degree at least $n-1$ in $w$, $z$ (recall that its degree matchess the class of $\Gamma$, which is consistent with the statement of Lemma~\ref{eq: complete Poncelet}). However, $P(w_0, w)=0$ could have other solutions $w\in \C \setminus \bbT$, and in consequence, the actual degree  of $P$ is $N-1$, with $N\ge n$.

Denote by $f_1, \dots, f_{N-1}$ the solutions (with account of multiplicity) of  
\begin{equation} \label{eq:foci}
P(0, w)=0.
\end{equation}

\begin{prop} \label{prop:fociBl}
The points $f_1, \dots, f_{N-1}$, which are solutions (with account of multiplicity) of  \eqref{eq:foci}, are the real foci of the curve $\Gamma$.
\end{prop}
\begin{proof}
	As it was pointed out,  the equation 
	$$
	g(u,v)=0
	$$
	of the dual curve (in the affine coordinates)  can be obtained from \eqref{eq:P=0} using the substitution \eqref{subst1}, with
	$$
	u=\frac{\zeta+\overline \zeta}{2}, \quad v=\frac{\zeta-\overline \zeta}{2i}.
	$$
	according to which $z=z(\zeta,\overline \zeta)$, $w=w(\zeta,\overline \zeta)$. The resulting equation should be homogenized to a curve in $\P^2(\C)$ with its equation obtained by taking $\zeta \to \zeta/t$, $t\in \C$. As it follows from \eqref{eq:dualG}, foci of $\Gamma$ are solutions of the equation
	$$
	g\left(\frac{1}{t}, \frac{i}{t}\right)=0. 
	$$
	We see that 
	$$
	u=\frac{\zeta+\overline \zeta}{2}=\frac{1}{t}, \quad v=\frac{\zeta-\overline \zeta}{2i}=\frac{i}{t}
	$$
	implies that $\zeta=0$ and $\overline \zeta=2/t$. (Notice  that here   $u$ and $v$ are  \emph{complex}  numbers and hence $\overline{\zeta}=u-iv$ need not 
	be the complex conjugate of $\zeta=u+iv$.) If we fix the branch of the square root in  \eqref{subst1} by $\sqrt{-1}=i$, then in this case $z=0$, $w=t$ so that the foci of $\Gamma$ are precisely the solutions of the equation
	$$
	P( 0, t)=0.
	$$ 
\end{proof}

An important consequence of the discussion in \cite[Section 5]{Mirman:2005dj} (see in particular Lemma 5.1.2) is the following result:
\begin{thm} \label{thm:Mirman}
Let $n\in \N$, $n\geq 3$. With the notation above (see also \eqref{notation:zeros}--\eqref{reversed}) and up to a multiplicative constant,
\begin{equation} \label{eq:representation}
	P(z,w)= \frac{w\,  \Phi_{N-1}(w) \Phi_{N-1}^*(z)- z\,  \Phi_{N-1}(z) \Phi_{N-1}^*(w)}{w-z},
\end{equation}
where
\begin{equation} \label{eq:representation2}
\Phi_{N-1}(z)=\Phi_{N-1}(z; f_1, \dots, f_{N-1}).
\end{equation}
Thus,  the real foci $f_j$ of $\Gamma$ are precisely the zeros of $\Phi_{N-1}(z)$. 

Furthermore, we have the   relation  
\begin{equation}\label{eq: relation}
N=n +2m+d,
\end{equation}
where $m$ is the number of real foci $f_j$  (accounted with multiplicity) that lie in the exterior of $\T$ 
and $d$ is the number of real foci $f_j$  (accounted with multiplicity) that lie on $\T$.

In particular, if  all real foci $f_j$ are  in $\bbD$, then $N=n$. 
\end{thm}

Since the right hand side in \eqref{eq:representation} is a Bezoutian of the polynomials $\Phi_{N-1}$ and $ \Phi_{N-1}^*$, we say that $P(z,w)$ is written in a \textit{Bezoutian form}. 

\begin{proof}
	Formula \eqref{eq:representation} has been proved in \cite{Mirman:2005dj}, see Theorem 1.1 and equation (20) therein. Relation \eqref{eq: relation} for $d=0$ also follows from \cite[Lemma 5.1.2]{Mirman:2005dj}. Thus, we only need to consider  $d>0$.

Suppose $f_j$ is a real focus such that $|f_j|=1$. Then $\overline{f_j} =1/f_j$  and we easily find that
$$
P(z,w)=\operatorname{const}\, (z-f_j)(w-\overline{f_j}) \widetilde{P}(z,w), 
$$
where  $\widetilde{P}(z,w)$ is the polynomial as given by \eqref{eq:representation} with $\Phi_{N-1}(z)$ replaced by $\widetilde{\Phi}_{N-2}(z):=\Phi_{N-2}(z;f_1,\ldots,f_{j-1},f_{j+1},\ldots,f_{N-1})$. If we had exactly $d$ foci on the unit circle, the degree of $\widetilde{P}(z,w)$ would be $N-d$, and by  \cite[Lemma 5.1.2]{Mirman:2005dj}, the number   of solutions of the equation $\widetilde P(w_0, w)=0$ on the unit circle satisfies $N-d = n + 2m$. Thus, \eqref{eq: relation} holds.
\end{proof}

\subsubsection{From the Bezoutian form to Poncelet curves}
\

As we have just seen, the Poncelet property of a curve implies that its Mirman's parametrization can be written in a Bezoutian form. This motivates us to analyze the converse: is  a Poncelet curve associated to any symmetric polynomial in a Bezoutian form? This approach is tempting since the only data necessary to generate such polynomials is the collections of foci $f_j$'s.

Let $P(z,w)$ be a symmetric polynomial in $z, w$, given by \eqref{eq:representation}--\eqref{eq:representation2}.
Clearly, representation \eqref{eq:representation} is  sufficient for property  \eqref{eq:P2=0} to hold. This does not mean however that $P(z,w)=0$ automatically parametrizes a Poncelet curve. A necessary condition for that is, for instance, that for every $z\in \bbT$, the equation $P(z,w)=0$ has the same number of solutions on $\bbT$. Mirman showed in \cite[Thm.~1]{Mirman:2003ho}  that this holds if $d=0$ and
\begin{equation} \label{eq:Mirmanscondition}
1+\sum_{j=1}^{N-1} \frac{1-|f_j|^2}{|z-f_j|^2} > 0 \quad \text{for all $z\in \bbT$}.
\end{equation}
More precisely, \eqref{eq:Mirmanscondition} implies that for each $z\in \T$, the equation $P(z,w)=0$ has exactly $n-1 =N-2m-1$ distinct solutions $w\in\T$ (notice that it is automatically satisfied if all $f_j\in \bbD$). 
For such polynomials $P(z,w)$ we can then define a family of $n$-Poncelet polygons by
$$
\mathscr{P}(z):=\partial \left(\operatorname{conv} \{w\in \T: P(z,w) =0\}\right)
$$
that generate the package of Poncelet curves $C_1,\dots,C_{[n/2]}$; here and in what follows we denote the convex hull of the set $S$ by $\operatorname{conv}(S)$.

Since $P(z,w)$ is a symmetric polynomial in $z$ and $w$, it can be written  in the form $P(z,w)=h(z+w,zw)$. If we set
$$
g(u,v):={\bar{\zeta}}^{N-1} h\left(\frac{2}{\bar{\zeta}},\frac{\zeta}{\bar{\zeta}}\right)\ \text{with } \zeta=u+iv, \; \overline \zeta=u-iv,
$$
then $g(u,v)$ is a polynomial with real coefficients of degree $N-1$.  Fujimura (see \cite{Fujimura:2018dt}) expressed $g(u,v)$ as a polynomial in  $\zeta$ 
and $\overline{\zeta}$, where the coefficients are given by explicit formulas   in the elementary symmetric functions (and their complex conjugates) of the $f_j$'s.
She only stated these formulas   in the case when $m=d=0$, but her arguments work  in  general.
Let $G(u_1,u_2,u_3)$ denote the homogenization of $g(u,v)$ as usual. Then the real algebraic curve given by $G(u_1,u_2,u_3)=0$ is the dual of a real  algebraic
curve $\Gamma$ of class $N-1$ whose real foci are the $f_j$'s. By construction, we   have 
$$
\bigcup_{k=1}^{[n/2]} C_k \subseteq \Gamma(\R).
$$
We cannot assure in general that this inclusion is not strict. 
Neither can we assure that the $C_k$'s will be Poncelet curves, in the sense of Definition~\ref{defPonc}. Moreover, Mirman \cite[Remark~1]{Mirman:2003ho} stated (without proof) that $C_1$ is always convex. In the following example we show that this is false in general, if $m>0$.

\begin{example} \label{exampleMirman}
Consider a family of examples with $P(z,w)$  given by \eqref{eq:representation}--\eqref{eq:representation2}, where $N=5$ and
$$
\Phi_{4}(z)=\Phi_{4}(z; 0,0,0,a), \quad a\in \C\setminus\bbT.
$$
In this case,  the corresponding polynomial $g(u,v)$ is of the form  
\begin{align*}
	g(u,v)&=(a^2-1) v^4+2\left(( a^2-1) u^2-2 a u-2a^2+6\right)v^2\\
	&\qquad + \left((a^2-1) u^4-8 a u^3+ (12-4 a^2 ) u^2+16 a u-16\right).
\end{align*}
It is easy to show that $g(u,v)$ is an irreducible polynomial (over $\C$) by looking at its Newton polygon which is the convex hull of the lattice points  $(0,0)$, $(4,0)$, and $(0,4)$. The method of showing (absolute) irreducibility of polynomials in several variables via Newton polytopes  is nicely explained in 
\cite{GAO2001501}.   Notice that $g(u,v)$ can be viewed as quadratic polynomial in $v^2$. Thus it is relatively easy to carry out explicit calculations to 
study  $\Gamma^*(\R)$. 

In particular, 
\begin{itemize}
	\item $\Gamma^*(\R)$ is  nonsingular and a disjoint union of two nested ``ovals''  in $\overline \C$.
	\item If $|a|<1$, then both   components of $\Gamma^*(\R)$ lie in the exterior of $\bbT $ (which means that $\Gamma(\R)\subset \bbD$). 
	Furthermore,  if $|a| \leq \sqrt{3-\sqrt{6}}\approx 0.741964$, then neither of the two   components of $\Gamma^*(\R)$ has  inflection points (so, $\Gamma(\R)$ has no cusps).
	If $ \sqrt{3-\sqrt{6}}< |a|<1$, then the larger  component of $\Gamma^*(\R)$ has  inflection points, see Figure~\ref{fig:counterex1}, left.
	
	\item If $1\le |a|\leq 5/3$, then one  component  of $\Gamma^*(\R)$ intersects  $\bbT$, see Figure~\ref{fig:counterex1}, right. 
	\item If $|a|>5/3$, then one   component  of $\Gamma^*(\R)$ lies in the exterior  and the other  
	lies in the interior  of $\bbT$ (so, the same is true for $\Gamma(\R)$).
	Furthermore,   if $|a|\geq \sqrt{3+\sqrt{6}}\approx 2.33441$, then neither of the two   components of $\Gamma^*(\R)$ has  inflection points, see Figure~\ref{fig:counterex2}.
	If $5/3<|a| < \sqrt{3+\sqrt{6}}$, then  the component of $\Gamma^*(\R)$ that lies in the exterior of $\T$   has  inflection points.
\end{itemize}

\begin{figure}[htb] 
	\begin{center}
		\begin{tabular}{c p{3mm}c}		
			\includegraphics[width=0.35\linewidth]{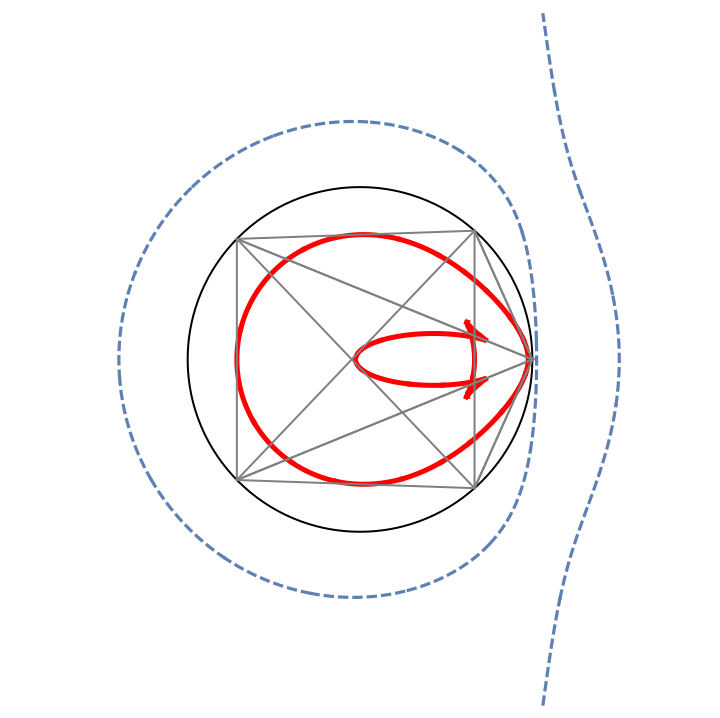} & 	& \includegraphics[width=0.35\linewidth]{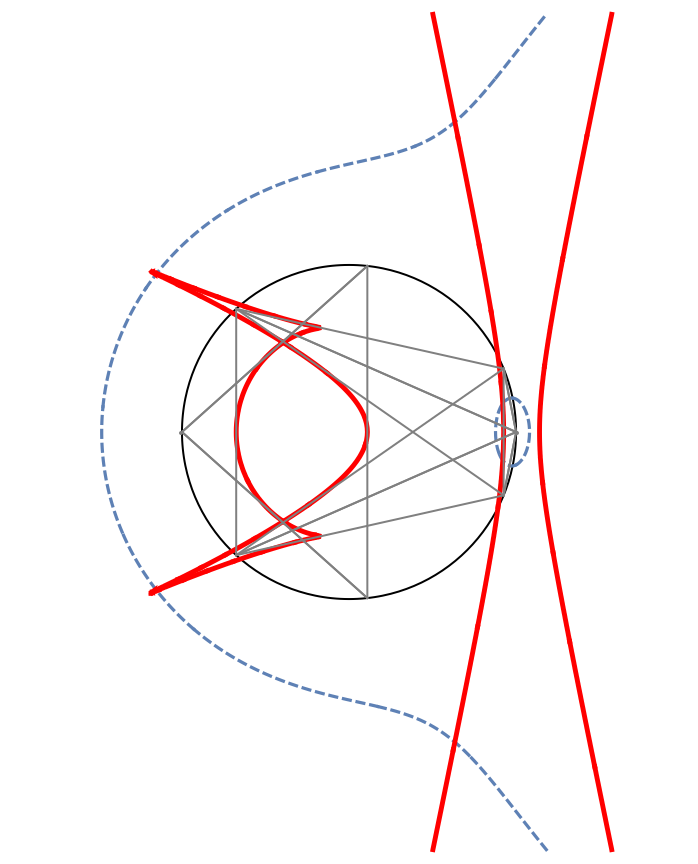} 
		\end{tabular}
	\end{center}
	\caption{Illustration for Example~\ref{exampleMirman}: curve $\Gamma(\R)$ (solid lines) and its dual $\Gamma^*(\R)$, dotted lines, for $a=0.9$ (left) and $a=1.5$, in which case Mirman's condition \eqref{eq:Mirmanscondition} is not satisfied (right). Notice that in this situation $\Gamma(\R)$ is tangent both to triangles and pentagons.}
	\label{fig:counterex1}
\end{figure} 
\begin{figure}[htb] 
	\begin{center}
		\begin{tabular}{c p{3mm}c}		
			\includegraphics[width=0.35\linewidth]{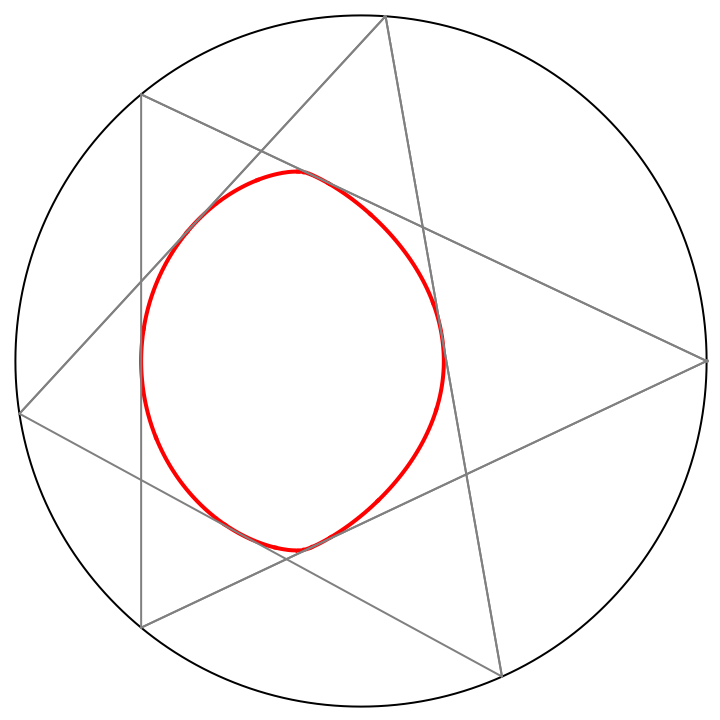} 
			& 	& \includegraphics[width=0.35\linewidth]{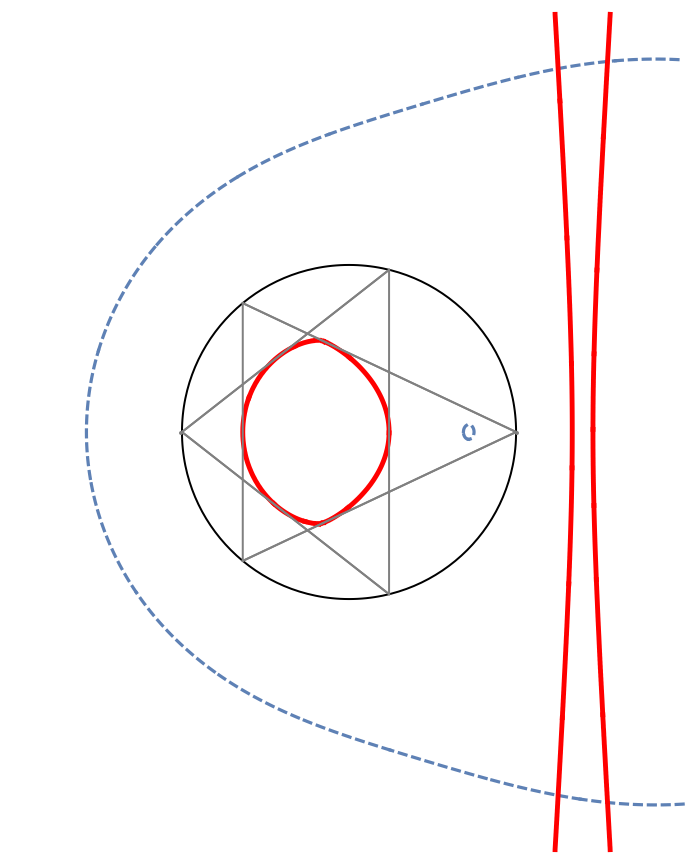} 
		\end{tabular}
	\end{center}
	\caption{Illustration for Example~\ref{exampleMirman}: a nonsingular $3$-Poncelet curve $C_1$ (left) and the corresponding algebraic curve $\Gamma(\R)$ (solid lines) and its dual $\Gamma^*(\R)$, dotted lines, for $a=2.4$. Curve $C_1$ for $a=2$ appears in Figure~\ref{fig:convexity}, right.}
	\label{fig:counterex2}
\end{figure} 

Notice that Mirman's condition \eqref{eq:Mirmanscondition} to generate Poncelet curves is satisfied whenever $|a|< 1$ or $|a|>5/3$.
For $|a|<1$,  we obtain a package of two $5$-Poncelet curves $C_1$ and $C_2$ with $C_1\cup C_2=\Gamma(\R)$. 
In this case, $C_1$ is nonsingular and equal to the boundary of the convex hull of $\Gamma(\R)$.
Furthermore, if $|a| \leq \sqrt{3-\sqrt{6}}$, then $C_2$ is also nonsingular and convex  since it's dual $C_2^*$ has no inflection points. 

For $|a|>5/3$,  the situation is more surprising and we obtain a package consisting of a single $3$-Poncelet curve $C_1$ with $C_1\not=\Gamma(\R)$, see Figure~\ref{fig:counterex2}.
If $|a|\geq \sqrt{3+\sqrt{6}}$, then $C_1$ is nonsingular and convex since it's dual $C_1^*$ has no inflection points. However, if $5/3<|a|<\sqrt{3+\sqrt{6}}$, then $C_1$ is singular since $C_1^*$  has inflection points which correspond to cusp 
singularities of $C_1$. 
\end{example}

\begin{example}
Another common misconception is that in the package of Poncelet curves $C_1$ encloses the rest of the curves $C_k$. This is definitely the case if the starting point is a family of convex Poncelet polygons $\{ \mathscr{P}(z):\, z\in \bbT\}$, as in  Definition~\ref{def:Package}. However, if we construct our curves from the representation \eqref{eq:representation}--\eqref{eq:representation2}, satisfying  \eqref{eq:Mirmanscondition}, this is not always true. The situation with $N=7$,
$$
\Phi_{6}(z)=\Phi_{6}(z; 0,0,0,0,0,a) 
$$
and $a=1.41$ is illustrated in Figure~\ref{fig:counterex3}.
\begin{figure}[htb] 
	\begin{center}
					\begin{overpic}[width=0.35\linewidth]{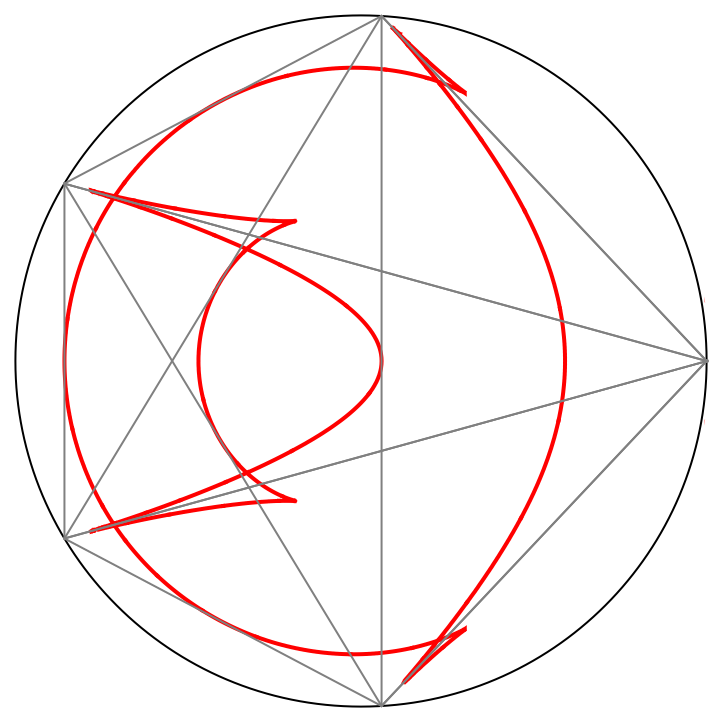}
			\put(95,107){$C_1$}
				\put(60,80){$C_2$}
		\end{overpic}
	\end{center}
	\caption{Illustration for Example~\ref{exampleMirman}: curves $C_1$ and $C_2$ have a non-empty intersection.}
	\label{fig:counterex3}
\end{figure} 
\end{example}

\section{Complete Poncelet curves with given foci} \label{sec:poncelet_general}

In this section, we gather all the knowledge accumulated so far to analyze minimal class complete $n$-Poncelet curves $\Gamma$ in the sense of the Definition~\ref{def:completeNponcelet}. In particular, we show that they are characterized by either one of these properties:
\begin{itemize}
	\item $\Gamma$ is of class $n-1$;
	\item all real foci of $\Gamma$ are inside the unit disk $\bbD$.
\end{itemize}
Moreover, in this case the set of foci $f_1, \dots, f_{n-1}$ determines $\Gamma$ completely, so that this is a bijection between points in $\bbD^{n-1}$ and complete $n$-Poncelet curves. The curve $\Gamma$ can be reconstructed from its real foci. We describe three approaches to this problem; all three can be formulated in terms of the paraorthogonal extension of the set $f_1,\dots, f_{n-1}$. 

Recall that the paraorthogonal extension (see Definition~\ref{def:popuc}) consists in applying the Szeg\H{o} recursion as in \eqref{def_POPUC} with $\Phi_{n-1}(z)=\Phi_{n-1}(z; f_1,\dots, f_{n-1})$ and obtaining the   $1$-parametric family of points $\mathcal{Z}_n^\lambda  = \{z_{n,1}^\lambda , \dots , z_{n,n}^\lambda \}$ on $\bbT$ as the zeros of  the resulting paraorthogonal polynomials of degree $n$. It turns out that sets $\mathcal{Z}_n^\lambda$ are identified by the same Blaschke product with zeros at $f_j$'s and, at the same time, are eigenvalues of a $1$-parametric family of unitary dilations (see Section~\ref{sec:dilations}) of the CMV matrix   $\cmv{n-1}\in \mathcal{S}_{n-1}$ associated with $\Phi_{n-1}$. In short, all these approaches are completely equivalent.

\begin{thm} \label{thm:matrices}
	Let $\Gamma$ be a plane real algebraic curve such that $\Gamma(\R)$ is a complete $n$-Poncelet curve ($n\ge 3$) with respect to $\bbT$, generated by a family of convex Poncelet polygons $\mathscr{P}(z)$ (see Definition~\ref{def:completeNponcelet}), so that
	$$
	\Gamma(\R)=	\bigcup_{k=1}^{[n/2]} C_k  .
	$$ 
Then  $\Gamma$ is of minimal class (that is, of class $n-1$) if and only if all real foci  are contained the unit disk $\bbD$. In this case,  $C_1$ is a convex curve, real foci $f_1, \dots, f_{n-1}$ of $\Gamma$ (enumerated with account of multiplicity)   determine $\Gamma$, and for every set of  points $f_1, \dots, f_{n-1}$  in $\mathbb D$, not necessarily all distinct,  there exists a (unique) algebraic $n$-Poncelet curve $\Gamma$ of class $n-1$ with real foci precisely at $f_1, \dots, f_{n-1}$. Such a curve $\Gamma$ is completely determined by its set of real points $\Gamma(\R)$.
		\item There are three equivalent realizations of the complete $n$-Poncelet curve $\Gamma$ of class $n-1$:
		\begin{enumerate}[(i)]
			\item $\Gamma(\R)$ is the envelope of the closed polygons supported on  the paraorthogonal extension (see Definition~\ref{def:popuc}) of the set $f_1,\dots, f_{n-1}$. 

			\item The Blaschke product 
			\begin{equation} \label{Blasch}
				B_{n}(z) = z\, \frac{\Phi_{n-1}(z)}{\Phi_{n-1}^*(z)}, \quad \Phi_{n-1}(z)=\Phi_{n-1}(z; f_1, \dots, f_{n-1}),
			\end{equation}
			identifies the set of points $\mathcal{Z}^\lambda=\{z_1^\lambda , \dots , z_n^\lambda \}$ on $\bbT$ (see Definition~\ref{def:identBlaschke}) in such a way that $\Gamma(\R)$ is the envelope of the closed polygons supported on  $\mathcal{Z}^\lambda$.
				\item there is a matrix $\bm A\in \mathcal{S}_{n-1}$ with its spectrum equal to the set $f_1,\dots, f_{n-1}$ (and hence, $\bm A$ unique up to unitary equivalence) such that 
			$$
			\operatorname{conv}(\Gamma(\R))=\partial W(\bm A).
			$$
			The equation of the dual curve $\Gamma^*$ in $\P^2(\C)$ is given by
			$$
			G_{\bm A}(u_1, u_2, u_3)=0,
			$$
			where $G_{\bm A}(u_1, u_2, u_3)$ was defined in  \eqref{defKippenhahn}. As a matrix $\bm A$ we can take the   cut-off CMV matrix  $\cmv{n-1}$ corresponding to the polynomial $\Phi_{n-1}(z; f_1,\dots, f_{n-1})$.
		\end{enumerate}
	Furthermore, each set $\mathcal{Z}^\lambda=\{z_1^\lambda , \dots , z_n^\lambda \}$ of eigenvalues of a $1$-parametric family of unitary dilations of  matrix $\bm A$ in (i) is identified by the Blaschke product \eqref{Blasch}.
\end{thm}
\begin{proof}
	We first show that  such a  curve $\Gamma$ of minimal class $n-1$ is completely determined by its set of real points $\Gamma(\R)$. 
	Let $\widetilde{\Gamma}$ be the real algebraic curve such that 
	$\widetilde{\Gamma}^*$  is the intersection of all real algebraic curves containing $\Gamma(\R)^*$.
	Note that $\widetilde{\Gamma}^*\subseteq \Gamma^*$ and $\widetilde{\Gamma}^*(\R)= \Gamma^*(\R)$.
	If $\widetilde{G}(u_1,u_2,u_3)$ and ${G}(u_1,u_2,u_3)$ are the homogenous polynomials of minimal degree defining  $\widetilde{\Gamma}^*$ and $\Gamma^*$,     respectively,
	then  $\widetilde{G}(u_1,u_2,u_3)$ divides  $G(u_1,u_2,u_3)$.
	By Lemma~\ref{lem:class}, the class of the curve $\widetilde{\Gamma}$ is $n-1$, which means that $\widetilde{G}(u_1,u_2,u_3)$ has degree $n-1$.
	Since $\widetilde{G}(u_1,u_2,u_3)$ divides  $G(u_1,u_2,u_3)$ and since $G(u_1,u_2,u_3)$ also has degree $n-1$, we  have 
	$G(u_1,u_2,u_3) = c \, \widetilde{G}(u_1,u_2,u_3)$, where $c$ is a nonzero constant.  Thus, $\Gamma^*=\widetilde{\Gamma}^*$
	and hence $\Gamma=\widetilde{\Gamma}$.  Since $\widetilde{\Gamma}$ is completely determined by $\Gamma(\R)$, so is $\Gamma$.
	
Let  $P(z,w)=0$ be Mirman's parametrization
of the complete $n$-Poncelet curve $\Gamma$, as explained in Section~\ref{sec:Bezoutian}. By Theorem~\ref{thm:Mirman},  $\Gamma$ is of class $n-1$ if and only if $N=n$, that is, $m=d=0$ (in other words, when all real foci are in $\bbD$).

Let $f_1, \dots, f_{n-1}$ be the real foci of a  complete $n$-Poncelet curve $\Gamma$ of class $n-1$. Again by Theorem~\ref{thm:Mirman}, its Mirman's parametrization is, up to a multiplicative constant,
 \begin{equation} \label{MP}
 P(z,w)= \frac{w\,  \Phi_{n-1}(w) \Phi_{n-1}^*(z)- z\,  \Phi_{n-1}(z) \Phi_{n-1}^*(w)}{w-z},
 \end{equation}
 with $\Phi_{n-1}$ defined in \eqref{Blasch}. In particular, it means that for any pairs of points $z, w \in \bbT$ satisfying $P(z,w)=0$ it holds that
 $$
 B_{n}(z) = B_{n}(w), \quad \text{with} \quad B_{n}(z)= z\, \frac{\Phi_{n-1}(z)}{\Phi_{n-1}^*(z)},
 $$
 i.e., $z, w$ are identified by $B_n$. 
 As we have seen in Section~\ref{sec:OPUC}, this is equivalent to being $z, w$ paraorthogonal extensions of $\Phi_{n-1}$.  Since $|B_{n}(z)|=1$ for $z\in \bbT$, there exists $\lambda\in \bbT$ such that
 $$
 z, w \in \mathcal{Z}^\lambda=\{z_1^\lambda , \dots , z_n^\lambda \},
 $$
 the set of solutions of $B_{n}(z)=\overline{\lambda}$. Hence,  by our discussion in Section~\ref{sec:cmv}, for each $\lambda \in\bbT$, points from $\mathcal{Z}^\lambda$ are also the eigenvalues of a rank one unitary dilation of the $(n-1)\times(n-1)$ cutoff CMV matrix $\bm A$ whose eigenvalues are $\{f_j\}_{j=1}^{n-1}$.  From \cite{MR3945586} (see also Section~\ref{sec:OPUC}) it follows that the unitary equivalence class of $\bm A$ is also uniquely determined by $f_1, \dots, f_{n-1}$ and as a representative of this class we can take $\bm A = \cmv{n-1}(\alpha_0, \dots, \alpha_{n-1})$, where the Verblunsky coefficients are determined by $ \Phi_{n-1}$
using the inverse Szeg\H{o} recursion: with notation \eqref{notation:zeros} and \eqref{notation:Verb},
 $$
 \Phi_{n-1}(z)=\Phi_{n-1}(z;f_1, \dots, f_{n-1})=\Phi_n^{(\alpha_0, \dots, \alpha_{n-1})}(z).
 $$

Let 
$$
\mathscr{P}(z):=\partial \left(\operatorname{conv} \{w\in \T: P(z,w) =0\}\right);
$$
by the previous observation, this family coincides with
$$
\left\{ \partial  \left( \text{conv} (\mathcal{Z}^\lambda )\right) : \lambda \in \bbT\right \} .
$$

Direct calculations (see also \cite[formula (6.10)]{MR2220032}) show that 
$$
\frac{d}{d\theta} \arg B\left( e^{i\theta}\right)=\frac{d}{d\theta} \arg \overline\lambda =1+\sum_{j=1}^{n-1} \frac{1-|f_j|^2}{|z-f_j|^2} > 0
$$
(compare with the necessary condition in \eqref{eq:Mirmanscondition}; see also an alternative expression in terms of orthonormal OPUC in \cite[formula (10.8)]{MR2383931}). The Inverse Function Theorem shows that  the Poncelet correspondence $\tau$ for $\{\mathscr{P}(z):\, z\in \bbT \}$ is strictly increasing, is smooth and satisfies \eqref{eq: monotonicity assumption}.
In consequence, the associated Poncelet curves $\{C_1,\ldots,C_{[n/2]}\}$ constructed in Section \ref{def:assocPoncelet},   make up the package of Poncelet curves generated by $\bm A$ in the sense described in \cite{MR3945586}.  Therefore, we may apply what is stated in \cite[page 130]{MR3945586}, namely that \eqref{MP}  is the Mirman parametrization of a complete $n$-Poncelet curve $\Gamma$.  It follows also that $\Gamma$ is an algebraic curve of class $n-1$ with real foci $\{f_j\}_{j=1}^{n-1}$ (see Proposition~\ref{prop:fociBl}). Our construction also shows the equivalence of (i) and (ii). 
 
It was proved in \cite{Gau:1998fk} that 
\begin{equation} \label{intersection}
	W(\bm A) =\bigcap_{\lambda \in \mathbb T} \text{conv} (\mathcal{Z}^\lambda ).
\end{equation}
This shows, in particular, that $\partial W(\bm A) $ coincides with the component $C_1$ in the package of Poncelet curves \eqref{Cunion}; our earlier considerations imply that $C_1$ is also the convex hull of $\Gamma(\R)$ (and thus, is a convex curve).  If $\Gamma'$ is the curve whose existence is guaranteed by Kippenhahn's Theorem (see Theorem~\ref{thm:Kippenhahn}), then since the convex hull of $\Gamma'(\R)$ is $\partial W(\bm A)$, it follows from Bezout's Theorem that $\Gamma=\Gamma'$.  The content of Section~\ref{sec:numran} implies that the equation for the dual curve of $\Gamma$ is 
\begin{equation} \label{eq:qA}
	G_{\bm A}(u_1, u_2, u_3)=0,
\end{equation}
where $G_{\bm A}$ was defined in  \eqref{defKippenhahn},  as claimed.  We have thus demonstrated that $\Gamma$ can be realized by the description in (iii).   The fact that each set  of eigenvalues of a $1$-parametric family of unitary dilations of    $\bm A$ in (iii) is identified by the Blaschke product $B_n$ is the direct consequence of \eqref{characteristicParaorthog}.

To show that $\Gamma$ is the unique curve with the desired properties, we suppose that $\widetilde  \Gamma$ is another complete $n$-Poncelet curve of class $n-1$ with real foci $\{f_j\}_{j=1}^{n-1}$.  Let $\widetilde{G}(u_1, u_2, u_3)=0$ be the equation of $\widetilde{\Gamma}^*$.  Then
\[
G_{\bm A}(1,i,f_j)=0=\widetilde{G}(1,i,f_j),\qquad\quad j=1,2\dots,n-1.
\]
Since $G_{\bm A}(1,i,z)$ and $\widetilde{G}(1,i,z)$ are both polynomials of degree (at most) $n-1$ and have $n-1$ zeros in common, it follows that they must be scalar multiples of one another and hence they define the same curve.  This means $\Gamma^*=\widetilde{\Gamma}^*$ and so $\Gamma=\widetilde{\Gamma}$.
\end{proof}

\begin{remark} \label{remark:CounterexampleMirman}
	\begin{enumerate}[a)]
		\item In Theorem~\ref{thm:matrices} we assume that $\Gamma$ is a complete $n$-Poncelet curve \textit{and} is of class $n-1$. We conjecture that this assumption is superfluous, and the results in Theorem~\ref{thm:matrices} can be established \textit{for any} complete Poncelet curve $\Gamma$. In particular, it would imply that all real foci of such a curve  are in $\bbD$. Notice that in general for a real algebraic curve $\Gamma$, the fact that $\Gamma(\R)\subset \bbD$ does not imply that its real foci  are in $\bbD$. A simple example \cite{counterexample} is the curve given by the equation  
		$$
	\left(  x_1^2 + x_2^2 -x_3^2/2\right)	\left(  (x_1-2)^2 + x_2^2 +x_3^2\right)=0
		$$
		with real foci at $(0:0:1)$ and $(2:0:1)$. 
		
\item It was established in  \cite{Mirman:2005dj} that the $n$-Poncelet property of a convex curve of class $n-1$ characterizes this curve as being a boundary of $W(\bm A)$ for some $\bm A\in \mathcal{S}_{n-1}$. The assumption on its class is essential as a counterexample in \cite[Example 1 on p. 131]{Mirman:2003ho} to a conjecture of Gau and Wu \cite[Conjecture 5.1]{Gau:1998fk} shows; see also the discussion in \cite{Gau:2003dw} on p. 184.

Curiously, the counterexample that appears in \cite{Mirman:2005dj}, constructed by violating the assumption that all $f_j$'s are in $\mathbb D$, is wrong. Namely, the authors consider the case of $\Phi_4(z; 0,0,0,a)=z^3(z-a)$, $a>1$, and the corresponding curve with Mirman's parametrization \eqref{eq:representation}. In  \cite{Mirman:2005dj}, they take $a=2$, but the resulting Poncelet curve is not convex! In fact, it is depicted in Figure~\ref{fig:convexity}, right. However, for a correct counterexample it is sufficient to use $a> \sqrt{3+\sqrt{6}}$, see our detailed discussion in Example~\ref{exampleMirman}. 

	\item Construction (ii) was extensively explored in the works \cite{MR3932079,Daepp:2010fz}. Its equivalence with (ii) was proved in \cite{MR3945586}.
	
	\item Property \eqref{intersection} characterizes the class $\mathcal{S}_{n-1}$: as it was shown in \cite[Theorem 4.4]{Gau:1998fk}, a  contraction $\bm A\in \C^{(n-1)\times (n-1)}$ (i.e. $\|\bm A\|\le 1$)  is in $\mathcal{S}_{n-1}$ if and only if  $W(\bm A)$ in $\mathbb{D}$ has the Poncelet property.
 	\end{enumerate}
\end{remark}

Although Theorem~\ref{thm:matrices} is stated for $n\ge 3$, its ``toy version'' for $n=2$ also holds:
\begin{prop}
Let $f\in \bbD$, and for $\lambda\in \bbT$ define
$$
z^\lambda_\pm = \frac{1}{2} \left(f - \overline{f  \lambda } \pm \sqrt{\left(f-  \overline{f  \lambda }\right)^2+4  \overline{ \lambda }} \right)  \in \bbT, 
$$
zeros of $\Phi_2^{(\overline f,  \lambda )}(z)=z(z-f)-\overline \lambda (1-\overline f z)$, or equivalently, eigenvalues of the $2\times 2$ cut-off CMV matrix
$$
\cmv{2}=\cmv{2}(\overline f,  \lambda)= \begin{pmatrix}
		f & \overline{\lambda} \sqrt{1-|f|^2}   \\ 
	\sqrt{1-|f|^2}  &- \overline{f \lambda }  
	\end{pmatrix}.
$$
Then for every  $\lambda\in \bbT$, the straight segment joining $z^\lambda_-$ and $z^\lambda_+$ passes through $f$. 
\end{prop}
See the proof for instance in \cite[Theorem 4.1]{MR3932079}.

 \section{Elliptic case} \label{sec:Mirman}
 
Recall the fundamental result proved by Darboux (see \cite{Darboux:1917},  also~\cite[Theorem 3]{Dragovic:2011jj}), mentioned as Theorem~\ref{thm:Darboux} in the Introduction:  if a component $C_j$ of a complete $n$-Poncelet curve \eqref{Cunion} is an ellipse and has $n$-Poncelet property (notice that here both $n$ are the same) then $C$ is a union of $[n/2]$ disjoint ellipses (also known as a package of  ellipses). Using the terminology introduced in Section \ref{sec:defs}, this is equivalent to assuming that  the curve $C_j$ is the envelope of polygons $\mathscr{P}_j(z)$ as defined in \eqref{defPoncPol}, with $\gcd(j,n)=1$. 

In this section we want to explore these ideas further and analyze the case of a component $C_j$ being an ellipse. This situation, especially when the convex component $C_1$ is assumed to be an ellipse, has attracted interest before, see e.g.~\cite{Brown:2004jl, Chien:2012bt,  Daepp:2002km, Daepp:2015bq, Daepp:2017fl, MR3932079, Gau:2006wr, Gau:2003by, Gorkin:2019ij, MR1322932, Mirman:2012ga}. 
As it could be expected, for an ellipse, Mirman's parametrization $P(z,w)=0$, described in Section~\ref{sec:evolvents}, has a more explicit form.

It was derived by   Mirman \cite{Mirman:2003ho, Mirman:2005hr} that  if $f_1, f_2\in \mathbb D$ are the foci of $C_j$ and $s$ is the length of its minor semiaxis, then  the straight line $\ell$ joining two points, $z,w\in \mathbb T$, is tangent to $C_j$ if and only if $z,w$ satisfy the equation $q(z,w)=0$, where 
  \begin{equation} \label{MirmanEllipse}
  q(z,w)=q(z,w;f_1,   f_{2}, s):=  \left(  w+b_1(z;f_1)\right)\left(  w+b_1(z;f_2)\right)-\frac{4s^2 z w}{\Phi_2^*(z; f_1, f_2)},
  \end{equation}
  see the notation in \eqref{notation:zeros}, \eqref{reversed}, \eqref{notation:Blaschke}.
  If $C_j$ is degenerated to a point ($f_1=f_2$, $s=0$), then we need to replace the expression above by
  \begin{equation} \label{MirmanEllipse2}
  q(z,w)=q(z,w;f_1,   f_{1}, 0)=q(z,w;f_1):=    w+b_1(z;f_1) .
  \end{equation}

In the non-degenerate case ($C_j$ is not a point) the Poncelet correspondence $\tau$ is correctly defined. Algebraically, it means that for $z\in \mathbb T$, the equation $q(z,w;f_1,   f_{2}, b)=0$ is quadratic in $w$ and has two  solutions,  the endpoints of the tangents to $C_j$ starting  at $z$ and ending on $\mathbb T$ (namely, $\tau(z)$ and $\tau^{-1}(z)$). This allows us to define an iterative process (that we call the \emph{circular Mirman's iteration}) as follows:
  	\begin{quote}
  		Start from $w_0\in \mathbb T$ and define $w_1\in \mathbb T$ as one of the two solutions of $q(w_0,w;f_1,   f_{2}, s)=0$.
  		
  		\noindent For $i=1, 2,\dots$,  choose as $w_{i+1}$ the solution of  $q(w_{i},w;f_1,   f_{2}, s)=0$ such that $w_{i+1}\neq w_{i-1}$.
  	\end{quote}
Clearly, $C_j$ has an $n$-Poncelet property if and only if $w_n=w_0$, and $n$ is the smallest natural number for which equality holds (in other words, if the orbit of the circular Mirman's iteration has length $n$).

The advantage of the algebraic interpretation of the iteration $\tau^k(z)$ is that we no longer need to assume $z\in \bbT$ (in which case the geometric interpretation is less obvious).

Reasoning as in the case of Mirman's parametrization (see Section~\ref{sec:MirmanPoncelet}), we see that replacing $w_0=0$ in $q(w_0,w;f_1,   f_{2}, s)=0$ yields the foci of $C_j$; in other words, starting in Mirman's iterations with $w_0=0$ necessarily yields either $w_1=f_1$ of $w_1=f_2$.  So, we can define the iterative process (let us call it the \emph{inner Mirman's iteration})  as follows:
\begin{quote}
	Start from $w_0=0$, define $w_1\in \{ f_1, f_2\} $ and for $i=1, 2,\dots$,  choose as $w_{i+1}$ the solution of  $q(w_{i},w;f_1,   f_{2}, s)=0$ such that $w_{i+1}\neq w_{i-1}$.
\end{quote}

Assume that $C_j$ has the $n$-Poncelet property and let $f_1,\dots, f_{n-1}$ be the set of points in $\bbD$ generated by this iterative process. According to Theorem~\ref{thm:matrices}, there exists a unique algebraic $n$-Poncelet curve $\Gamma$ of class $n-1$ with real foci precisely at $f_1, \dots, f_{n-1}$, and let $P(z,w)=0$ be its Mirman's parametrization. The following result was proved in \cite{Mirman:2003ho} and is a simple consequence of  Bezout's theorem:
\begin{thm}\label{thm:MirmanEllipses}
	Under the assumptions above, $q(z,w;f_1,   f_{2}, s)$ divides $P(z,w)$. Moreover, in the decomposition \eqref{Cunion}, each component $C_k$ is an ellipse, and they are all disjoint. The ellipse $C_{[n/2]}$ is degenerate (a point) if $n$ is even. If we denote by $f_{2k-1}$ and $f_{2k}$ the foci of $C_k$, and by $s_k$ its minor semiaxis, then  
	\begin{equation}\label{factorizat}
		P(z,w;f_1, \dots, f_{n-1})= \prod_{k=1}^{[n/2]} q(z,w;f_{2k-1},   f_{2k}, s_k);
	\end{equation}
	if $n$ is even, we take $f_{n-1}=f_{n}$, $s_{[n/2]}=0$.
\end{thm}
Comparing this statement with Theorem \ref{thm:Darboux} in the Introduction, we see that Mirman in fact reproved Darboux's result while being unaware of his work!

\begin{remark}
An assumption of Theorem \ref{thm:MirmanEllipses} is that the curve $\Gamma$ is of class exactly $n-1$, as the following construction shows. 
With $f_1,\dots, f_{n-1}$ being the points in $\bbD$ generated by inner Mirman's iterations for $C_j$, let 
$$
	b_{n }(z)=	b_{n }(z;0,  f_1,\dots, f_{n-1})=   \frac{z \Phi_{n-1}(z)}{\Phi_{n-1}^*(z)}, \quad \Phi_{n-1}(z)= \Phi_{n}(z; f_1,\dots, f_{n-1}).
$$
Pick arbitrary points $g_1,\dots, g_{m-1}\in \bbD$, define 
$$
	B_{m }(z)=	b_{m}(z; 0,  g_1,\dots, g_{m-1})=   \frac{z \Phi_{m-1}(z)}{\Phi_{m-1}^*(z)}, \quad \Phi_{m-1}(z)= \Phi_{m-1}(z; g_1,\dots, g_{m-1}),
$$
and consider the composition 
$$
D_{mn}(z)= (B_m \circ b_n)(z )=B_m(b_n(z)),
$$
which is again a Blaschke product with $D_{mn}(0)=0$. Then the envelope of polygons supported on solutions of $D_{mn}(z)=\overline{\lambda}$, $\lambda\in \bbT$, is a complete $mn$-Poncelet curve $\widetilde \Gamma$ of class $mn-1$. One of its components is the $n$-Pocelet ellipse $C_j$. However, as we have seen, the inner Mirman's iteration for $C_j$  generates only $n-1$ values, $f_1,\dots, f_{n-1}$, among them, the two foci of this ellipse. All these values are a 
subset of the real foci of $\widetilde \Gamma$, and by considerations above and by Theorem \ref{thm:MirmanEllipses}, these will be foci of a package of $[n/2]$ ellipses, all components of $\widetilde \Gamma$.  However, this does not mean that \emph{the whole} curve $\widetilde \Gamma$ is a package of Poncelet ellipses. 
\end{remark}

\begin{example} \label{exampleEllipses}
Let $n=3$, $f_1, f_2\in \bbD$, $g_1,\dots, g_{m-1}\in \bbD$ ($m\ge 2$), and as above, 
$$
D_{3m}(z)= (B_m \circ b_3)(z ), \quad 	b_{3 }(z)=	b_{3 }(z;0,  f_1,  f_{2}), \quad B_{m }(z)=	b_{m}(z; 0,  g_1,\dots, g_{m-1}).
$$
Then   the envelope of polygons supported on solutions of $D_{3n}(z)=\overline{\lambda}$, $\lambda\in \bbT$,   contains a  $3$-Poncelet ellipse with foci $f_1$ and $f_2$, whose minor semiaxis is
$$
s^2=\frac{1-|f_1|^2-|f_2|^2+|f_1f_2|^2}{4}
$$ 
(use \eqref{factorizat}  with $n=3$, $m=1$).  Notice that the inner Mirman iteration for that ellipse yields only $f_1$ and $f_2$. 
\begin{figure}[htb] 
	\begin{center}
		\begin{overpic}[width=0.35\linewidth]{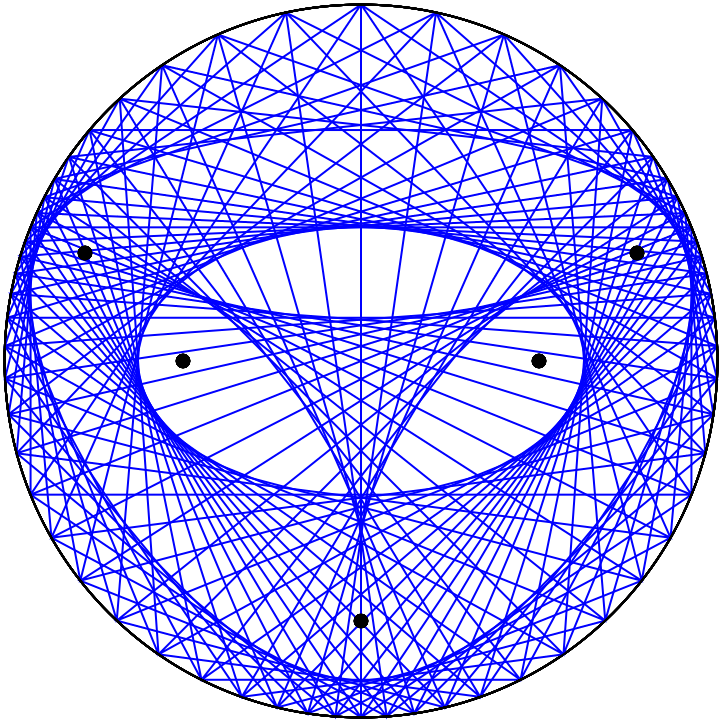}
		\end{overpic}
	\end{center}
	\caption{Illustration for Example~\ref{exampleEllipses}, with $f_1=-f_2=1/2$ and $g_1=i/2$ ($m=2$).}
	\label{fig:FigExampleBlaschke}
\end{figure} 
\end{example}

A simple consequence of our considerations is the following
\begin{thm}
	Let $\Gamma$ be a complete $n$-Poncelet curve of class $n-1$, $C_j$ is an elliptic component of the decomposition \eqref{Cunion}, and assume that the maximal set generated by inner Mirman's iteration corresponding to $C_j$  contains only $m-1$ distinct values $w_1, \dots, w_{m-1}$,  $m\le n$. Then
	\begin{itemize}
		\item $n$ is divisible by $m$;
		\item $\{ w_1, \dots, w_{m-1} \} \subset \{ f_1, \dots, f_{n-1} \} $, and $w_j$ are foci of a package of $[m/2]$ ellipses, all of them components of $\Gamma$.
	\end{itemize}
\end{thm}
For instance, if $n$ is prime, then the existence of any elliptic component in $\Gamma(\R)$  forces all $C_k$'s to be ellipses (it is a package of ellipses).

We can reformulate Mirman's iteration by observing the independent term in \eqref{MirmanEllipse} and applying Vieta's formulas: given a value $z$, the two solutions $w'$ and $w''$ of $q(z,w;f_1,   f_{2}, s)=0$ satisfy
  $$
  w' w''= b_1(z;f_1)b_1(z;f_2)=b_2(z;f_1,f_2).
  $$
  In particular, in the circular Mirman iteration,
  \begin{equation}\label{iterationC}
  w_{i+1}w_{i-1}=b_2(w_i;f_1,f_2), \quad i=1, 2, \dots 
  \end{equation}
  This is a three term recurrence relation that has an advantage of not requiring knowledge of value $s$ of the semiaxis. It needs two initial values, for which we could use $w_0$ and $w_1$. However, notice that it is necessary to know $s$ in order to calculate $w_1$. Obviously, since all $w_i$'s are on $\mathbb T$, formulas \eqref{iterationC} generate $w_{i+1}$ from $(w_{i-1}, w_{i})$ for all $i\in \N$.

  As before, relaxing the assumptions we still can use the iterations \eqref{iterationC}; however, in the present situation there are two additional moments to address:
  \begin{itemize}
  	\item The iteration breaks down when $w_{j-1}=0$;
  	\item for that reason, we cannot use  $w_0=0$ and $w_1=f_1$ (or $w_1=f_2$)  to initialize the process; we need to calculate a third value, a solution of $q(w_{1},w;f_1,   f_{2}, s)=0$, $w_2\neq 0$. Notice that
  	$$
  	q(f_i ,w;f_1,   f_{2}, s)=w\left(   w+b_1(f_i;f_j)- \frac{4s^2 f_i }{\Phi_2^*(f_i; f_1, f_2)} \right) , \quad i\neq j, \quad i, j\in \{ 1,2\},
  	$$
  	so  we can initialize the inner Mirman iterations explicitly by
  	\begin{equation} \label{initIter}
  	w_1 =f_i\in \{ f_1, f_2\}, \quad w_2=  -b_1(f_i;f_j) + \frac{4s^2 f_i }{\Phi_2^*(f_i; f_1, f_2)} , \quad i\neq j, \quad i, j\in \{ 1,2\}.
  	\end{equation}
  	As it was observed, the value of $s$ is necessary to calculate $w_2$. 
  \end{itemize}  

We further explore the  case of ellipses with Poncelet properties in a forthcoming paper \cite{FociEllipses}.

\section*{Acknowledgments}

The second author was partially supported by Simons Foundation Collaboration Grants for Mathematicians (grant 710499) and by the European Regional Development Fund along with Spanish Government (grant MTM2017-89941-P) and Junta de Andaluc\'{\i}a (grant UAL18-FQM-B025-A, as well as research group FQM-229 and Instituto Interuniversitario Carlos I de F\'{\i}sica Te\'orica y Computacional), and by the University of Almer\'{\i}a (Campus de Excelencia Internacional del Mar  CEIMAR).

The fourth author graciously acknowledges support from Simons Foundation Collaboration Grant for Mathematicians 707882.

The authors are indebted to V.~Dragović,  J.~Langer and D.~Singer for helpful discussions.



%

\end{document}